\documentclass[11pt]{article}
\usepackage{amssymb}
\usepackage{amsmath}
\usepackage{amsthm}
\usepackage[textwidth=16cm, textheight=25.5cm]{geometry}

\usepackage[english]{babel}
\usepackage[utf8]{inputenc}
\usepackage[T1]{fontenc}
\usepackage{xcolor}
\usepackage[pdfborder={0 0 0}]{hyperref}
\usepackage{booktabs}
\newtheorem{thm}{Theorem}[section]
\newtheorem{cor}[thm]{Corollary}
\newtheorem{lem}[thm]{Lemma}
\newtheorem{prop}[thm]{Proposition}
\newtheorem{conjecture}[thm]{Conjecture}
\theoremstyle{definition}
\newtheorem{defn}[thm]{Definition}
\newtheorem{remark}[thm]{Remark}
\newtheorem{example}[thm]{Example}
\newtheorem*{assumption}{Assumption}

\usepackage{etoolbox}
\def\DPut#1#2{\global\csdef{D#1}{#2}#2}%
\def\DDef#1#2{\global\csdef{D#1}{#2}}%

\begin{document}
\def\leq{\leqslant}%
\def\geq{\geqslant}%

\def\spann#1{\operatorname{span}\left\langle#1\right\rangle}%
\def\tensor{\otimes}
\def\onto{\twoheadrightarrow}
\def\into{\hookrightarrow}
\def\hook{\,\lrcorner\,}
\def\Apolar#1{\operatorname{Apolar}\left(#1\right)}%
\def\gr{\operatorname{gr}}%
\def\DP{P}%
\def\DS{S}%
\def\Dx{\alpha}%
\def\Ann#1{\operatorname{Ann}(#1)}%
\DDef{grp}{\mathbb{G}}%

\title{Classifying local Artinian Gorenstein algebras}%
\author{Joachim Jelisiejew\thanks{The author is a doctoral fellow at the Warsaw Center of
Mathematics and Computer Science financed by the Polish program KNOW.
This article is supported by Polish National Science Center, project
2014/13/N/ST1/02640.
Contact information: Joachim Jelisiejew, Faculty of Mathematics, Informatics
and Mechanics, University of Warsaw. \textbf{Email:} jjelisiejew@mimuw.edu.pl}}%
\maketitle%
\vspace{-8mm}
\begin{abstract}
    The classification of local Artinian Gorenstein algebras is equivalent to the study
    of orbits of a certain non-reductive group action on a polynomial ring.
    We give an explicit formula for the orbits and their tangent spaces. We
    apply our technique to analyse when an algebra is isomorphic to its
    associated graded algebra.
    We classify
    algebras with Hilbert function $(1, 3, 3, 3, 1)$, obtaining finitely many
    isomorphism types, and those with Hilbert function $(1, 2, 2, 2, 1, 1, 1)$.
    We consider fields of arbitrary, large enough, characteristic.
\end{abstract}
\vspace*{-1mm}
{\small\textbf{MSc classes:} 13H10 (Primary), 13N10, 14C05 (Secondary)}
\vspace*{-4mm}

\tableofcontents

\vspace{-3mm}

\section{Introduction}

    The problem of classifying Artinian Gorenstein algebras up to isomorphism is hard and long
    studied, a comprehensive reference is~\cite{iakanev}.
    The results usually rest on additional assumptions: small
    length~\cite{Poonen, Casnati_Notari_6points, cn09,
    Casnati__Isomorphisms_types_of_aG_to_nine, cn10} or being a codimension two complete
    intersection~\cite{briancon, EliasVallaAlmostStretched, EliasHoms}. See
    also~\cite{ia_deformations_of_CI, Ehrenborg, Eastwood_Isaev,
    Fels_Kaup_nilpotent_and_homogeneous} for other approaches.

    Let $k$ be a field. So far we do not impose any conditions such as $k$ being
    algebraically closed or on its characteristic (see
    Section~\ref{sec:applications}), however the reader
    may freely assume that $k = \mathbb{C}$.

    Let $(A, \DPut{mmA}{\mathfrak{m}_A}, k)$ be a local Artinian Gorenstein
    $k$-algebra. As the notation suggests, throughout the paper we assume that $k\to A/\DmmA$ is an
    isomorphism. This assumption is automatic if $k$ is algebraically closed.
    It is standard and it is used to discard examples such as $k =
    \mathbb{R} \subset  A = \mathbb{C}$.

    An important numerical invariant of $A$ is the Hilbert function $\DPut{HA}{H_A}(i) =
    \dim_k \DmmA^i/\DmmA^{i+1}$. The \emph{socle degree} $d$ is the maximal number such that
    $H_A(d)\neq 0$. If $\DHA(d-i) = \DHA(i)$ for all $i$, then the associated
    graded algebra $\DPut{grA}{\operatorname{gr}A} = \bigoplus_{i\geq 0}
    \DmmA^i/\DmmA^{i+1}$ is also Gorenstein, see \cite[Proposition 1.7]{ia94}.

    We may now formulate our motivating problems:
    \begin{enumerate}
        \item What is the classification up to isomorphism of local Artinian Gorenstein algebras with Hilbert function
            $H_A$?
        \item If $H_A$ is symmetric, what are the sufficient conditions for
            $A \simeq \DgrA$?
    \end{enumerate}
    The first of these problems naturally leads to the second: suppose we are classifying
    $A$ whose Hilbert function $H_A$ is symmetric. Since $\DgrA$ is in this
    case Gorenstein, one natural way is to first determine possible
    $\DgrA$ and then classify $A$ having fixed $\DgrA$.
    Note that it is convenient to classify graded Gorenstein algebras by using
    ideas from projective geometry: Waring, border and smoothable ranks, and
    secant varieties, see~\cite{geramita_inverse_systems_of_fat_points,
    iakanev} for the overview of methods and~\cite{LO, bubu2010} for the state
    of the art on secant varieties.
    See~\cite{Migliore_Hvectors, StanleyCombinatoricsAndCommutative} and
    references therein for information about the possible Hilbert functions.
    See~\cite[Appendix]{ia94} for a list of possible Hilbert functions for
    small values of $\dim_k\, A$.

    Our main goal is to refine, simplify and make explicit the existing theory. Having done this,
    we are gratified with classification results in several cases:
    \begin{enumerate}
        \item We reprove and extend the result of Elias and Rossi:
            \begin{prop}[\cite{EliasRossiShortGorenstein,
                EliasRossi_Analytic_isomorphisms},
            Corollary~\ref{ref:eliasrossi:cor}]\label{main:eliasrossi}
                Let $k$ be a field of characteristic other than $2$ or $3$.
            A local Artin Gorenstein $k$-algebra with Hilbert function
            $(1, n, n, 1)$ or Hilbert function $\left(1, n, \binom{n+1}{2}, n, 1\right)$ is
            isomorphic to its associated graded algebras.
            \end{prop}
            The original statement is proven in \cite{EliasRossiShortGorenstein,
            EliasRossi_Analytic_isomorphisms} for fields of characteristic
            zero. In Example~\ref{ex:1nn1char2} we show that the statement is not true for fields of
            characteristic two or three.

            In Proposition~\ref{main:eliasrossi} the condition on the
            Hilbert function can be rephrased by saying that the algebra is
            \emph{compressed}, see~Section~\ref{sec:compressed}. The algebras appearing
            in Proposition~\ref{main:eliasrossi} are precisely compressed algebras of
            socle degree three and four.
            A analogous statement for
            compressed algebras of higher socle degree is false.
            However, we show that compressed algebras of all socle degrees are
            near to the graded algebras, see Corollary~\ref{ref:compressedpoly:cor}.

        \item We prove that general algebras with Hilbert
            function $(1, 2, 3, 3, 2, 1)$ and general algebras with Hilbert
            function $(1, 2, 2, 2,  \ldots , 2, 1)$ are isomorphic to
            their associated graded algebras
            (see~Example~\ref{ex:123321general} or
            Example~\ref{ex:122lotsof21} for a precise meaning of the word
            general).
        \item An algebra is \emph{canonically graded} if it isomorphic to
            its associated graded algebra.
            We investigate the set of (dual socle generators of) algebras
            which are canonically graded. In
            Proposition~\ref{ref:nopennclosed:prop} we show that this set is
            irreducible, but in general neither open nor closed in the
            parameter space. This answers a question of Elias and
            Rossi~\cite[Remark~3.6]{EliasRossi_Analytic_isomorphisms}.

            We investigate density of the set canonically graded algebras
            and in Proposition~\ref{ref:generalcangradrarely:prop} we prove
            the \emph{only if} part of the following conjecture.

            \begin{conjecture}[Conjecture~\ref{ref:density:conj}]
                Assume that $k$ is a field of characteristic not equal to $2, 3, 5$.
                Say that an algebra $A$ has \emph{type} $(n, d)$ if $H_A(1) =
                n$ and the socle degree of $A$ is $d$.
                Then a general algebra of type $(n,
                d)$ is canonically graded if and only if $(n, d)$ belongs to
                the following list:
                \begin{enumerate}
                    \item $d\leq 4$ and $n$ arbitrary,
                    \item $d = 5$ and $n\leq 6$,
                    \item $d = 6$ and $n = 2$.
                    \item $d$ arbitrary and $n = 1$.
                \end{enumerate}
            \end{conjecture}
        \item We classify algebras with Hilbert function $(1, 3, 3, 3, 1)$,
            obtaining eleven isomorphism types, see Example~\ref{ex:13331}.
            We also discuss the case of $(1, 3, 4, 3, 1)$, show that there are
            infinitely many isomorphism types and classify
            algebras which are not canonically graded, see
            Example~\ref{ex:13431}. Non-canonically graded
            algebras with Hilbert function $(1, 3, 4, 3, 1)$ are investigated
            independently
            in~\cite{masuti_rossi_Artinian_level_algebras_of_socle_degree_four}.
        \item We classify algebras with Hilbert function $(1, 2, 2, 2, 1, 1,
            1)$ obtaining $|k| + 1$ isomorphism types, see
            Example~\ref{ex:1222111}. This example was
            worked out by Elias and Valla~\cite{EliasVallaAlmostStretched}.
    \end{enumerate}

    We now explain our approach.
    In short, the classification problem may be reduced to representation
    theory, see~\cite{Emsalem}. Assume for a moment for
    simplicity of presentation that the characteristic of $k$ is zero. Let $n =
    H_A(1)$. Then $A$ may be seen as a quotient of a fixed power series ring $S =
    k[[\Dx_1, \ldots ,\Dx_n]]$. The ring $S$ acts on a polynomial ring $\DP =
    k[x_1, \ldots ,x_n]$ by
    letting $\Dx_i$ act as the partial derivative with respect to $x_i$.
    It turns out that $A = \DS/\Ann{f}$ for an element $f\in \DP$, called the
    \emph{dual socle generator} of $A$.
    Then the problem of classifying algebras reduces to
    classifying quotients of $\DS$, which in turn boils down to
    classifying $f\in \DP$.
    There is a certain group $\Dgrp$ acting on $\DP$ such that
    $\DS/\Ann{f}  \simeq \DS/\Ann{g}$ if and only if $f$ and $g$ lie in the
    same $\Dgrp$-orbit. We summarise this (for fields of arbitrary characteristic) in
    Proposition~\ref{ref:Gorenstein:prop}.

    Our contribution to the above is in providing explicit formulas for all objects
    involved and in applying Lie-theoretic ideas. In particular, we give formulas for the $\Dgrp$-orbit and its
    tangent space, thus we are able to compute dimensions of orbits. We also
    find the unipotent radical of $\Dgrp$ and analyse its orbits, which are
    closed.

    The explicit formulas make several known results, such as
    those by Elias and Rossi, corollaries of the presented
    theory. This is, in our opinion, the most important contribution of this
    paper. Accordingly, we have tried to keep the paper as elementary as possible.

    The reader willing to get a quick idea of the contents of this paper should
    analyse Example~\ref{ex:1nn1}, referring to previous results if necessary.
    Otherwise we recommend to read through the paper linearly.
    The paper is organized as follows.
    First, we recall and develop the abstract theory of inverse systems; in
    particular we give an explicit formula for the automorphism group action in
    Proposition~\ref{ref:dualautomorphism:prop}. Then we explain the link
    between classifying algebras and elements of $\DP$ in
    Proposition~\ref{ref:Gorenstein:prop}. We give some examples, in
    particular Example~\ref{ex:1nn1}, and then discuss the
    basic Lie theory in Section~\ref{ssec:liegroups}.
    Finally, in Section~\ref{sec:applications} we present the applications
    mentioned above.
    Apart from the conjecture above, there are several natural questions that
    are in our opinion worth considering:
    \begin{enumerate}
        \item What are $H$ such that there are finitely many isomorphism types
            of Artinian Gorenstein algebras with Hilbert function $H$? What is
            the classification in these cases?
        \item An Artinian Gorenstein algebra is \emph{rigid} if it cannot be
            deformed to a non-isomorphic algebra, see~\cite{Schlessinger}.
            The $\Dgrp$-orbit of a rigid algebra should be large.
            For $k$ algebraically closed, are
            there rigid $k$-algebras other than $k$?
        \item Given the Hilbert function, are there non-trivial bounds on the
            dimensions of $\Dgrp$-orbits of algebras with this Hilbert
            function?
        \item Can the above ideas of classification be generalised effectively
            to non-Gorenstein algebras?
    \end{enumerate}

\subsection*{Acknowledgements}

I thank Pedro Macias Marques, Alessandra Bernardi and Kristian Ranestad for
the joint work which inspired the explicit formula in
Proposition~\ref{ref:dualautomorphism:prop}.
I thank Jaros\l{}aw Buczy\'nski for helpful suggestions
and Aldo Conca for pointing some other possible applications.
I thank Anthony Iarrobino for comments on a near-final version of the
paper and the referees for suggesting many improvements.
This paper has emerged from a desire to understand the surprising results of
Joan~Elias and Maria Evelina Rossi \cite{EliasRossiShortGorenstein,
EliasRossi_Analytic_isomorphisms}. I thank them for their hospitality in
Barcelona and Genova.

The computations of apolar algebras were done using Magma~\cite{Magma} but could also be done
using Macaulay2 \cite{M2} command \texttt{fromDual} or
 Joan Elias library~\cite{Elias_CAS} for Singular~\cite{Singular}.

\section{Preliminaries and theoretical results}

\subsection{Power series ring $\DS$ and its dual $\DP$}
\def\DmmS{\mathfrak{m}}%
\def\Ddual#1{{#1}^{\vee}}%
\def\ip#1#2{\left\langle #1, #2\right\rangle}%
\def\ord#1{\operatorname{ord}\left(#1\right)}%
\def\Homthree#1#2#3{\operatorname{Hom}_{#1}(#2, #3)}%
\def\myN{\mathbb{N}_{0}}%
In this section we introduce our main objects of study: the power series ring
$\DS$ and its action on the dual divided power (or polynomial) ring $\DP$.

By $\myN$ we denote the set of non-negative integers.
Let $k$ be an algebraically closed field of arbitrary characteristic.
Let $\DS$ be a power series ring over $k$ of dimension $\dim \DS = n$ and let
$\DmmS$ be its maximal ideal.
By $\ord{\sigma}$ we denote the order of a non-zero $\sigma\in \DS$ i.e.~the largest $i$
such that $\sigma\in \DmmS^i$. Then $\ord{\sigma} = 0$ if and only if
$\sigma$ is invertible.
Let $\Ddual{\DS} = \Homthree{k}{\DS}{k}$ be the space of functionals on $\DS$.
We denote the pairing between $\DS$ and $\Ddual{\DS}$ by
\[
    \ip{-}{-}: \DS \times \Ddual{\DS} \to k.
\]

\begin{defn}\label{ref:contraction:propdef}
    The dual space $\DP \subset \Ddual{\DS}$ is the linear subspace of functionals
    eventually equal to zero:
    \[
        \DP = \left\{ f\in \Ddual{\DS}\ |\ \forall_{D\gg 0}\ \ip{\DmmS^D}{f} = 0\right\}.
    \]
    On $\DP$ we have a structure of $\DS$-module via precomposition:
    for every $\sigma\in \DS$ and $f\in \DP$ the element $\sigma\hook f\in
    \DP$ is defined via the equation
    \begin{equation}\label{eq:hookvsip}
        \ip{\tau}{\sigma\hook f} = \ip{\tau\sigma}{f}\quad \mbox{for
        every}\ \tau\in \DS.
    \end{equation}
    This action is called \emph{contraction}.
\end{defn}

    Existence of contraction is a special case of the following
    construction, which is basic and foundational for our approach.
    Let $L:\DS \to \DS$ be a $k$-linear map. Assume that there exists an
    integer $s$ such that $L(\DmmS^{i}) \subset \DmmS^{i+s}$ for all $i$.
    Then the dual map $\Ddual{L}:\Ddual{\DS}\to \Ddual{\DS}$ restricts to
    $\Ddual{L}:\DP \to \DP$. Explicitly, $\Ddual{L}$ is given by the equation
    \begin{equation}
        \ip{\tau}{\Ddual{L}(f)} = \ip{L(\tau)}{f}\quad \mbox{for
        every}\ \tau\in \DS,\ f\in \DP.
        \label{eq:dual}
    \end{equation}
    To obtain contraction with respect to $\sigma$ we use $L(\tau) =
    \sigma\tau$, the multiplication by $\sigma$. Later in this paper we will also
    consider maps $L$ which are automorphisms or derivations.

\def\aa{\mathbf{a}}%
\def\bb{\mathbf{b}}%
\def\xx{\mathbf{x}}%
\def\DPel#1#2{#1^{[#2]}}%
To get a down to earth description of $\DP$, choose $\Dx_1, \ldots
,\Dx_n\in \DS$ so that $\DS = k[[\Dx_1, \ldots ,\Dx_n]]$.
Write $\Dx^{\aa}$ to denote $\Dx_1^{a_1} \ldots
\Dx_n^{a_n}$. For every $\aa\in \myN^{n}$ there is a unique element
$\DPel{\xx}{\aa}\in \DP$ dual to $\Dx^{\aa}$, given by
\[
    \ip{\Dx^{\bb}}{\DPel{\xx}{\aa}} = \begin{cases}
        1 & \mbox{if }\aa = \bb\\
        0 & \mbox{otherwise.}
    \end{cases}
\]
Additionally, we define $x_i$ as the functional dual to $\Dx_i$, so that $x_i
= \DPel{\xx}{(0, \ldots 0, 1, 0,  \ldots , 0)}$ with one on $i$-th position.

Let us make a few remarks:
\begin{enumerate}
    \item The functionals $\DPel{\xx}{\aa}$ form a basis of $\DP$.
    \item The contraction action is given by the formula
        \[
            \Dx^{\aa} \hook \DPel{\xx}{\bb} =
            \begin{cases}
                \DPel{\xx}{\bb - \aa} & \mbox{if } \bb\geq \aa,\mbox{ that is, }\forall_{i}\ b_i \geq a_i\\
                0 & \mbox{otherwise}.
            \end{cases}
        \]
        Therefore our definition agrees with the standard,
        see~\cite[Definition~1.1, p.~4]{iakanev}.
\end{enumerate}
We say that $\DPel{\xx}{\aa}$ has degree $\sum a_i$. We will freely speak
about constant forms, linear forms, (divided) polynomials of bounded degree
etc.

We endow $\DP$ with a topology, which is the Zariski topology of an affine space.
It will be used when speaking about general polynomials and closed orbits, but
for most of the time it is not important.

Now we will give a ring structure on $\DP$. It will be used
crucially in Proposition~\ref{ref:dualautomorphism:prop}.
For multi-indices $\aa, \bb\in \myN^{n}$ we define $\aa! = \prod (a_i!)$, $\sum
\aa = \sum a_i$ and $\binom{\aa + \bb}{\aa} = \prod_i
\binom{a_i + b_i}{a_i} = \binom{\aa + \bb}{\bb}$.
\begin{defn}
    We define multiplication on $\DP$ by
    \begin{equation}\label{eq:divpowmult}
        \DPel{\xx}{\aa} \cdot \DPel{\xx}{\bb} := \binom{\aa + \bb}{\aa}
        \DPel{\xx}{\aa + \bb}.
    \end{equation}
    In this way $\DP$ is a \emph{divided power ring}.
\end{defn}
The multiplicative structure on $\DP$ can be defined in a coordinate-free
manned using a
natural comultiplication on $\DS$. We refer to \cite[\S A2.4]{Eisenbud} for
details in much greater generality.
\begin{example}\label{ex:dividedpows}
    Suppose that $k$ is of characteristic $3$. Then $\DP$ is not isomorphic to
    a polynomial ring. Indeed, $x_1\cdot x_1\cdot x_1 =
    \left(2\DPel{x_1}{2}\right)\cdot x_1 = 3\DPel{x_1}{3} = 0$. Moreover $\DPel{x_1}{3}$
    is not in the subring generated by $x_1, \ldots ,x_n$.
\end{example}

Note that linear forms from $\DS$ act on $\DP$ as derivatives. Therefore we
can interpret $\DS$ as lying inside the ring of differential operators on $\DP$.
We will need the following related fact.
\def\ithpartial#1#2{#1^{(#2)}}%
\begin{lem}\label{ref:commutator:lem}
    Let $\sigma\in \DS$ and denote by $\ithpartial{\sigma}{i}$ its
    $i$-th partial derivative. For every $f\in \DP$ we have
    \begin{equation}\label{eq:weylalgebra}
        \sigma\hook (x_i\cdot f) - x_i\cdot (\sigma\hook f) =
        \ithpartial{\sigma}{i}\hook f.
    \end{equation}
\end{lem}

\begin{proof}
    Since the formula is linear in $\sigma$ and $f$ we may assume these are
    monomials. Let $\sigma = \Dx_i^{r}\tau$, where $\Dx_i$ does not appear in
    $\tau$. Then $\ithpartial{\sigma}{i} = r\Dx_i^{r-1}\tau$.
    Moreover $\tau\hook (x_i \cdot f) = x_i\cdot (\tau\hook f)$, thus we
    may replace $f$ by $\tau\hook f$ and reduce to the case $\tau = 1$,
    $\sigma = \Dx_i^{r}$.

    Write $f = \DPel{x_i}{s}g$ where $g$ is a monomial in variables other
    than $x_i$. Then $x_i\cdot f = (s+1)\DPel{x_i}{s+1} g$ according
    to~\eqref{eq:divpowmult}.
    If $s+1 < r$ then both sides of~\eqref{eq:weylalgebra} are
    zero. Otherwise
    \[
        \sigma\hook (x_i \cdot f) = (s+1)\DPel{x_i}{s+1-r}g,\ \
        x_i\cdot (\sigma\hook f) = x_i\cdot\DPel{x_i}{s-r}g =
        (s-r+1)\DPel{x_i}{s-r+1}g,\ \ \ithpartial{\sigma}{i}\hook f =
        r\DPel{x_i}{s-(r-1)}g,
    \]
    so Equation~\eqref{eq:weylalgebra} is valid in this case also.
\end{proof}

\begin{remark}
    Lemma~\ref{ref:commutator:lem} applied to $\sigma = \Dx_i$ shows that
    $\Dx_i\hook (x_i\cdot f) - x_i\cdot (\Dx_i\hook f) = f$. This can be
    rephrased more abstractly by saying that $\Dx_i$ and $x_i$
interpreted as linear operators on $\DP$ generate a Weyl algebra.
\end{remark}

Example~\ref{ex:dividedpows} shows that $\DP$ with its ring structure has
certain properties distinguishing it from the polynomial ring, for example it contains nilpotent elements.
Similar phenomena do not occur in degrees lower than the characteristic or in
characteristic zero, as we show in Proposition~\ref{ref:divpowssmallarepoly:prop} and
Proposition~\ref{ref:divpowerispoly:prop} below.
\begin{prop}\label{ref:divpowssmallarepoly:prop}
    Let $\DP_{\geq d}$ be the linear span of $\{\DPel{\xx}{\aa}\ |\  \sum \aa \geq d\}$.
    Then $\DP_{\geq d}$ is an ideal of $\DP$, for all $d$.
    Let $k$ be a field of characteristic $p$.
    The ring $\DP/\DP_{\geq p}$ is isomorphic to the truncated polynomial ring. In fact
    \[
        \Omega:\DP/\DP_{\geq p}\to k[x_1, \ldots ,x_n]/(x_1, \ldots ,x_n)^p
    \]
    defined by
    \[
        \Omega\left( \DPel{\xx}{\aa} \right) = \frac{x_1^{a_1}  \ldots
        x_n^{a_n}}{a_1! \ldots a_n!}.
    \]
    is an isomorphism.
\end{prop}
\begin{proof}
    Since $\Omega$ maps a basis of $\DP/I_p$ to a basis of $k[x_1, \ldots
    ,x_n]/(x_1, \ldots ,x_n)^p$, it is
    clearly well defined and bijective. The fact that $\Omega$ is a $k$-algebra homomorphism
    reduces to the equality $\binom{\aa + \bb}{\aa} = \prod \frac{(a_i +
    b_i)!}{a_i!b_i!}$.
\end{proof}

\paragraph{Characteristic zero case.}
In this paragraph we assume that $k$ is of characteristic zero.
This case is technically easier, but there are two competing conventions:
contraction and partial differentiation. These agree up to an
isomorphism. The main aim of this section is clarify this isomorphism and
provide a dictionary between divided power rings used in the paper and
polynomial rings in characteristic zero.
Contraction was already defined above, now we define the action of $\DS$ via
partial differentiation.
\begin{defn}\label{ref:partialdiff:def}
    Let $k[x_1, \ldots ,x_n]$ be a polynomial ring. There is a (unique) action
    of $\DS$ on $k[x_1, \ldots ,x_n]$ such that the element $\Dx_i$ acts a
    $\frac{\partial}{\partial x_i}$. For $f\in k[x_1, \ldots ,x_n]$ and
    $\sigma\in \DS$ we denote this action as $\sigma\circ f$.
\end{defn}

The following Proposition~\ref{ref:divpowerispoly:prop} shows that in characteristic zero the ring $\DP$ is in fact
polynomial and the isomorphism identifies the $\DS$-module structure on $\DP$
with that from Definition~\ref{ref:partialdiff:def} above.
\begin{prop}\label{ref:divpowerispoly:prop}
    Suppose that $k$ is of characteristic zero. Let $k[x_1, \ldots ,x_n]$ be a
    polynomial ring with $\DS$-module structure as defined
    in~\ref{ref:partialdiff:def}. Let $\DPut{stale}{\Omega}:\DP \to k[x_1,
    \ldots ,x_n]$ be defined via
    \[
        \Dstale\left(\DPel{\xx}{\aa}\right) = \frac{x_1^{a_1}  \ldots x_{n}^{a_n}}{a_1! \ldots
        a_n!}.
    \]
    Then $\Dstale$ is an isomorphism of rings and an isomorphism of $\DS$-modules.
\end{prop}
\begin{proof}
    The map $\Dstale$ is an isomorphism of $k$-algebras by the same argument as in
    Proposition~\ref{ref:divpowssmallarepoly:prop}.
     We leave the check that $\Dstale$ is a $\DS$-module
    homomorphism to the reader.
\end{proof}

Summarising, we get the following corresponding notions.
\begin{center}

\begin{tabular}{@{}l l l @{}}
    Arbitrary characteristic && Characteristic zero \\ \midrule
    divided power series ring $\DP$ && polynomial ring $k[x_1, \ldots ,x_n]$\\
    $\DS$-action by contraction (precomposition) denoted $\sigma\hook f$  && $\DS$ action by
    derivations denoted $\sigma\circ f$\\
    $\DPel{\xx}{\aa}$ && $\xx^{\aa}/\aa!$\\
    $x_i = \DPel{\xx}{(0, \ldots 0, 1, 0,  \ldots , 0)}$ && $x_i$\\
\end{tabular}
\end{center}

\subsection{Automorphisms and derivations of the power series ring}

Let as before $\DS = k[[\Dx_1, \ldots ,\Dx_n]]$ be a power series ring with
maximal ideal $\DmmS$. This ring has a huge automorphism group: for every
choice of elements $\sigma_1, \ldots ,\sigma_n\in \DmmS$ whose images span
$\DmmS/\DmmS^2$ there is a unique
automorphism $\varphi:\DS\to \DS$ such that $\varphi(\Dx_i) = \sigma_i$.
Note that $\varphi$ preserves $\DmmS$ and its powers. Therefore the dual
map $\Ddual{\varphi} : \Ddual{\DS}\to \Ddual{\DS}$ restricts to
$\DPut{phid}{\Ddual{\varphi}}:\DP\to \DP$.
The map $\Dphid$ is defined (using the pairing of
Definition~\ref{ref:contraction:propdef}) via the condition
\begin{equation}\label{eq:definitionofdual}
    \ip{\varphi(\sigma)}{f} = \ip{\sigma}{\Dphid(f)}\quad\mbox{for all
    }\sigma\in \DS,\ f\in \DP.
\end{equation}

\def\DDD{\mathbf{D}}%
Now we will describe this action explicitly.
\begin{prop}\label{ref:dualautomorphism:prop}
    Let $\varphi:\DS\to\DS$ be an automorphism. Let $D_i = \varphi(\Dx_i) -
    \Dx_i$. For $\aa\in \myN^{n}$ denote $\DDD^{\aa} = D_1^{a_1} \ldots
    D_n^{a_n}$.
    Let $f\in \DP$.
    Then
    \[
        \Ddual{\varphi}(f) = \sum_{\aa\in \myN^n} \DPel{\xx}{\aa}\cdot \left(
        \DDD^{\aa} \hook f \right) = f + \sum_{i=1}^n x_i\cdot (D_i
        \hook f) +  \ldots .
    \]
\end{prop}

\begin{proof}

    We need to show that $\ip{\sigma}{\Ddual{\varphi}(f)} =
    \ip{\varphi(\sigma)}{f}$ for all $\sigma\in \DS$. Since $f\in \DP$, it is
    enough to check this for all $\sigma\in k[\Dx_1, \ldots ,\Dx_n]$. By
    linearity, we may assume that $\sigma = \Dx^{\aa}$.

    For every $g\in \DP$ let $\varepsilon(g) = \ip{1}{g}\in k$.
    We have
\def\parenthese#1{\left( #1 \right)}%
    \[
        \ip{\varphi(\sigma)}{f} = \ip{1}{\varphi(\sigma)\hook f} =
        \varepsilon\parenthese{\varphi(\sigma)\hook f} =
        \varepsilon\parenthese{\sum_{\bb\leq \aa} \binom{\aa}{\bb}\left(\Dx^{\aa -
        \bb}\DDD^{\bb}\right)
    \hook f} = \sum_{\bb\leq \aa}
    \varepsilon\parenthese{\binom{\aa}{\bb}\Dx^{\aa -
    \bb}\hook (\DDD^{\bb}\hook f)}.
    \]
    Consider a term of this sum. Observe that for every $g\in \DP$
    \begin{equation}\label{eq:aux}
        \varepsilon\parenthese{\binom{\aa}{\bb}\Dx^{\aa - \bb}\hook g} =
        \varepsilon\parenthese{\Dx^{\aa} \hook \left( \DPel{\xx}{\bb}\cdot
        g \right)}.
    \end{equation}
    \def\cc{\mathbf{c}}%
    Indeed it is enough to check the above equality for $g = \DPel{\xx}{\mathbf{c}}$
    and both sides are zero unless $\cc = \aa - \bb$, thus it is enough
    to check the case $g = \DPel{\xx}{\aa - \bb}$, which is straightforward.
    Moreover note that if $\bb \not\leq \aa$, then the right hand side is zero
    for all $g$, because $\varepsilon$ is zero for all $\DPel{\xx}{\cc}$ with
    non-zero $\cc$.
    We can use~\eqref{eq:aux} and remove the restriction $\bb\leq \aa$,
    obtaining
    \begin{align*}
        \sum_{\bb} \varepsilon\parenthese{\Dx^{\aa}\hook
        \left(\DPel{\xx}{\bb}\cdot (\DDD^{\bb}\hook f)\right)} =
        \varepsilon\parenthese{ \sum_{\bb} \Dx^{\aa}\hook
        \left(\DPel{\xx}{\bb}\cdot (\DDD^{\bb}\hook f)\right)} =
        \ip{\Dx^{\aa}}{ \sum_{\bb}\DPel{\xx}{\bb}\cdot (\DDD^{\bb}\hook f)}
        =\\
        \ip{\Dx^{\aa}}{\Ddual{\varphi}(f)} = \ip{\sigma}{\Ddual{\varphi}(f)}.
    \end{align*}
\end{proof}

Consider now a derivation $D:\DS\to \DS$. Here and elsewhere a derivation is a $k$-linear map satisfying
the Leibnitz rule. The derivation $D$ gives rise to a dual map
$\Ddual{D}:\DP \to \DP$. We wish to describe it explicitly.
\begin{prop}\label{ref:dualderivation:prop}
    Let $D:\DS\to \DS$ be a derivation and $D_i := D(\Dx_i)$.
    Let $f\in \DP$. Then
    \[
        \Ddual{D}(f) = \sum_{i=1}^n x_i\cdot (D_i\hook f).
    \]
\end{prop}

\begin{proof}
    The proof is similar, though easier, to the proof of
    Proposition~\ref{ref:dualautomorphism:prop}.
\end{proof}

\begin{remark}\label{ref:lowersdegree:rmk}
    Suppose $D:\DS\to \DS$ is a derivation such that $D(\DmmS) \subseteq
    \DmmS^2$. Then $\deg(\Ddual{D}(f)) < \deg(f)$. We say that $D$ lowers
    the degree.
\end{remark}

\paragraph{Characteristic zero case.}
Let $k$ be a field of characteristic zero. By $\xx^{\aa}$ we denote the
monomial $x_1^{a_1} \ldots x_n^{a_n}$ in the polynomial ring $k[x_1, \ldots
,x_n]$. Then, in the notation of Proposition~\ref{ref:divpowerispoly:prop}, we
have
\[\Dstale\left(\DPel{\xx}{\aa}\right) = \frac{1}{\aa!}\xx^{\aa}.\]
Clearly, an automorphism of $\DS$ gives rise to an linear map $k[x_1, \ldots
,x_n]\to k[x_1, \ldots ,x_n]$.
We may restate Proposition~\ref{ref:dualautomorphism:prop} and
Proposition~\ref{ref:dualderivation:prop} as
\begin{cor}\label{ref:dualautcharzero:cor}
    Let $\varphi:\DS\to\DS$ be an automorphism. Let $D_i = \varphi(\Dx_i) -
    \Dx_i$. For $\aa\in \myN^{n}$ denote $\DDD^{\aa} = D_1^{a_1} \ldots
    D_n^{a_n}$.
    Let $f\in k[x_1, \ldots ,x_n]$.
    Then
    \[
        \Ddual{\varphi}(f) = \sum_{\aa\in \myN^n} \frac{\xx^\aa}{\aa!} \left(
        \DDD^{\aa} \circ f \right) = f + \sum_{i=1}^n x_i (D_i
        \circ f) +  \ldots .
    \]
    Let $D:\DS\to \DS$ be a derivation and $D_i := D(\Dx_i)$. Then
    \[
        \Ddual{D}(f) = \sum_{i=1}^n x_i(D_i\circ f).
    \]
\end{cor}

\begin{example}
    Let $n = 2$, so that $\DS = k[[\Dx_1, \Dx_2]]$  and consider a linear map $\varphi:\DS \to \DS$ given by
    $\varphi(\Dx_1) = \Dx_1$ and $\varphi(\Dx_2) = \Dx_1 + \Dx_2$. Dually,
    $\Ddual{\varphi}(x_1) = x_1 + x_2$ and $\Ddual{\varphi}(x_2) = x_2$.
    Since $\varphi$ is linear, $\Ddual{\varphi}$ is an automorphism of
    $k[x_1, x_2]$.
    Therefore $\Ddual{\varphi}(x_1^3) = (x_1 + x_2)^3$.
    Let us check this equality using Proposition~\ref{ref:dualautcharzero:cor}.
    We have $D_1 = \varphi(\Dx_1) - \Dx_1 = 0$ and $D_2 = \varphi(\Dx_2) - \Dx_2 = \Dx_1$.
    Therefore $\DDD^{(a, b)} = 0$ whenever $a > 0$ and $\DDD^{(0, b)}
    = \Dx_1^b$.

    We have
    \[
        \Ddual{\varphi}\left( x_1^3 \right) = \sum_{(a, b)\in \mathbb{N}^2}
        \frac{x_1^{a}x_2^{b}}{a!b!}(\DDD^{(a, b)} \circ x_1^3) = \sum_{b\in
            \mathbb{N}} \frac{x_2^{b}}{b!}(\Dx_1^{b} \circ x_1^3) = x_1^3 +
            \frac{x_2}{1} \cdot(3x_1^2) + \frac{x_2^2}{2} \cdot(6x_1) +
            \frac{x_2^3}{6}\cdot(6) = (x_1 + x_2)^3,
    \]
    which indeed agrees with our previous computation.

    When $\varphi$ is not linear, $\Ddual{\varphi}$ is not an endomorphism of
    $k[x_1, x_2]$ and computing it directly from definition becomes harder. For
    example, if $\varphi(\Dx_1) = \Dx_1$ and $\varphi(\Dx_2) = \Dx_2 +
    \Dx_1^2$, then
    \[
        \Ddual{\varphi}(x_1) = x_1,\quad \Ddual{\varphi}(x_1^4) = x_1^4 + 12x_1^2x_2 + 12x_2^2.
    \]
\end{example}

\subsection{Local Artinian Gorenstein algebras}\label{ssec:AGA}

Let $(A, \DmmA, k)$ be a local Artinian Gorenstein algebra. Recall that the \emph{socle degree} of
$A$ is defined as $d =\max\left\{ i\ |\ \DmmA^i\neq 0\right\}$. The \emph{embedding
dimension} of $A$ is $H_A(1) = \dim \DmmA/\DmmA^2$. Note that $A$ may be
presented as a quotient of a power series ring $\DS$ over $k$ if and only $\dim \DS \geq
H_A(1)$. Therefore it is most convenient to analyse the case when $\dim \DS =
H_A(1)$, which we will usually do.

Recall that by Macaulay's inverse systems, for every local Gorenstein quotient $A =
\DS/I$ there is an $f\in \DP$ such that $I = \Ann{f}$, see~\cite{iakanev,
ia94}, \cite[Chapter 21]{Eisenbud} and~\cite[Chapter~IV]{Macaulay_inverse_systems}. Such an $f$ is called
a dual socle generator of $A$. Given $I$, the choice of $f$ is non-unique, but
any two choices differ by a unit in $\DS$. Conversely, for $f\in \DP$ we denote
\[\Apolar{f} = \DS/\Ann{f},\]
which is called the \emph{apolar algebra} of $f$ and $\Ann{f}$ is called the
\emph{apolar ideal} of $f$. It is important that the apolar algebra of the
top degree form of $f$ is a quotient of $\operatorname{gr} A$,
see~\cite[Chapter 1]{ia94}.
We may compare $A$ and $f$ as
follows:
\begin{center}

\begin{tabular}{@{}l l l @{}}
    algebra $A$ && module $\DS f$ \\ \midrule
    $\DS$-module $A$ & $\simeq$ & $\DS f$\\
    $\dim_k A$ &=& $\dim_k \DS f$\\
    socle degree of $A$ &=& $\deg(f)$\\
    certain quotient of $\operatorname{gr} A$ &$ \simeq $&
    $\Apolar{f_{\deg(f)}}$\\
    If for all $i$ we have &\\
    $H_A(d-i) = H_A(i)$ then $\operatorname{gr} A$ &$ \simeq $& $\Apolar{f_{\deg(f)}}$
\end{tabular}
\end{center}

$\DDef{Aut}{\operatorname{Aut}(\DS)}$%
As mentioned in the introduction, our interest is in determining when
two local Artinian Gorenstein quotients of the power series ring are isomorphic.
The following Proposition~\ref{ref:Gorenstein:prop} connects this
problem with the results of the previous section. It is well-known, see
e.g.~\cite{Emsalem}.
Let $\DS^*$ denote the group of invertible elements of $\DS$ and let
\[\Dgrp := \DAut \ltimes \DS^*\] be the group generated by $\DAut$ and
$\DS^*$ in the group of linear operators on $\DS$. As the notation suggests, the group $\Dgrp$ is a semidirect
product of those groups: indeed $\varphi\circ \mu_s \circ \varphi^{-1} =
\mu_{\varphi(s)}$, where $\varphi$ is an automorphism, $s\in \DS$ is invertible and $\mu_{s}$ denotes
multiplication by $s$. By Equation~\eqref{eq:dual} we have an action of $\Dgrp$ on
$\DP$. Here $\DS^*$
acts by contraction and $\DAut$ acts as described in Proposition~\ref{ref:dualautomorphism:prop}.
\def\Ann#1{\operatorname{Ann}\left( #1 \right)}%
\def\orbit#1{\Dgrp\!\cdot\!#1}%
\begin{prop}\label{ref:Gorenstein:prop}
    Let $A = \DS/I$ and $B = \DS/J$ be two local Artinian Gorenstein algebras.
    Choose $f, g\in \DP$ so that $I = \Ann{f}$ and $J = \Ann{g}$.
    The following conditions are equivalent:
    \begin{enumerate}
        \item\label{it:isom} $A$ and $B$ are isomorphic,
        \item\label{it:conj} there exists an automorphism $\varphi:\DS\to \DS$ such that
            $\varphi(I) = J$,
        \item\label{it:dualconj} there exists an automorphism $\varphi:\DS\to \DS$ such that
            $\Ddual{\varphi}(f) = \sigma\hook g$, for an
            invertible element $\sigma\in \DS$.
        \item\label{it:grp} $f$ and $g$ lie in the same $\Dgrp$-orbit of $\DP$.
    \end{enumerate}
\end{prop}

\begin{proof}

    Taking an isomorphism $A  \simeq B$, one obtains $\varphi':\DS\to
    B = \DS/I$, which can be lifted to an automorphism of $\DS$ by choosing
    lifts of linear forms. This proves $\ref{it:isom}.\iff \ref{it:conj}$.

    $\ref{it:conj}.\iff\ref{it:dualconj}$. Let $\varphi$ be as
    in~\ref{it:conj}. Then $\Ann{\Ddual{\varphi}(f)} = \varphi(\Ann{f}) =
    \varphi(I) = J$. Therefore the principal $\DS$-submodules of $\DP$ generated by
    $\Ddual{\varphi}(f)$ and $g$ are equal, so that there is an invertible
    element $\sigma\in \DS$ such that $\Ddual{\varphi}(f) = \sigma\hook g$.
    The argument can be reversed.

    Finally, \ref{it:grp}.~is just a reformulation of
    \ref{it:dualconj}.
\end{proof}

The invertible element in Point~\ref{it:dualconj} cannot be discarded, see
Example~\ref{ex:1nn1} and Example~\ref{ex:13331}. The outcome of this theorem is that we are interested
in the orbits of elements of $\DP$ under the $\Dgrp$ action.

\begin{example}[quadrics]\label{ref:examplequarics:ex}
    Let $f\in P_2$ be a quadric of maximal rank. Then
    $\orbit{f}$ is the set of (divided) polynomials of degree two, whose quadric part is
    of maximal rank.
\end{example}

\begin{example}[compressed cubics]\label{ex:1nn1}
    Assume that $k$ has characteristic not equal to two.

    Let $f = f_3 + f_{\leq 2}\in\DP_{\leq 3}$ be an element of degree three
    such $H_{\Apolar{f}} = (1, n, n, 1)$, where $n = H_{\DS}(1)$. Such a
    polynomial $f$ is called
    compressed, see \cite{iaCompressed, EliasRossi_Analytic_isomorphisms} or
    Section~\ref{sec:compressed} for a definition.

    We claim that there is an element $\varphi\in\Dgrp$
    such that $\Ddual{\varphi}(f_3) = f$. This implies that $\Apolar{f}
    \simeq \Apolar{f_{3}} = \gr\Apolar{f}$. We say that the
    apolar algebra of $f$ is \emph{canonically graded}.

    Let $A = \Apolar{f_3}$.
    Since $H_A(2) = n = H_{\DS}(1)$, every linear form in $\DP$ may be obtained as
    $\delta\hook f$ for some operator $\delta\in \DS$ of \emph{order two}, see
    Remark~\ref{ref:socleminusone:rmk} below.
    We pick operators $D_1, \ldots ,D_n$ so that $\sum x_i \cdot (D_i
    \hook f_3) = f_2 + f_1$.
    Explicitly, $D_i$ is such that $D_i\hook f_3 = (\Dx_i \hook f_2)/2 +
    \Dx_i\hook f_1$. Here we
    use the assumption on the characteristic.

    Let $\varphi:\DS\to \DS$ be an automorphism defined via $\varphi(\Dx_i) =
    \Dx_i + D_i$. Since $(D_i D_j)\hook f = 0$ by degree reasons, the explicit formula in
    Proposition~\ref{ref:dualautomorphism:prop} takes the form
    \[
        \Ddual{\varphi}(f_3) = f_3 + \sum x_i \cdot (D_i \hook f_3) = f_3 +
        f_2 + f_1.
    \]
    The missing term $f_0$ is a constant, so that we may pick an order three
    operator $\sigma\in \DS$ with $\sigma\hook \Ddual{\varphi}(f_3) = f_0$.
    Then $(1 + \sigma)\hook (\Ddual{\varphi}(f_3)) = f_3 + f_2 + f_1 + f_0 =
    f$, as claimed.

    The isomorphism $\Apolar{f} \simeq \Apolar{f_3}$ was first proved by Elias and Rossi, see
    \cite[Thm~3.3]{EliasRossiShortGorenstein}. See
    Example~\ref{ex:1nn1Lie} for a more conceptual proof for $k = \mathbb{C}$ and
    Section~\ref{sec:compressed} for generalisations. In particular
    Corollary~\ref{ref:eliasrossi:cor} extends this example to cubic and
    quartic forms.
\end{example}

\begin{remark}[Graded algebras]\label{ref:graded:rem}
    In the setup of Proposition~\ref{ref:Gorenstein:prop} one could specialize
    to homogeneous ideals $I$, $J$ and homogeneous polynomials $f, g\in \DP$.
    Then Condition 1.~is equivalent to the fact that $f$ and $g$ lie in
    the same $\operatorname{GL}(\DmmS/\DmmS^2)$-orbit. The proof of
    Proposition~\ref{ref:Gorenstein:prop} easily restricts to this case, see~\cite{geramita_inverse_systems_of_fat_points}.
\end{remark}

Now we turn our attention to derivations. The notation and motivation come from Lie theory,
see section on characteristic zero below. Let $\DPut{Der}{\mathfrak{aut}}$
denote the space of derivations of $\DS$ preserving $\DmmS$, i.e.~derivations such that $D(\DmmS) \subseteq
\DmmS$. Let $\DS \subset \Homthree{k}{\DS}{\DS}$ be given by
sending $\sigma\in \DS$ to multiplication by $\sigma$.
Let
\[
    \DPut{grptang}{\mathfrak{g}} := \DDer  + \DS,
\]
where the sum is taken in the space of linear maps from $\DS$ to $\DS$.
Then $\Dgrptang$ acts on $\DP$ as defined in Equation~\eqref{eq:dual}.
\def\perpspace#1{\left( #1 \right)^{\perp}}%
\def\tang#1{\Dgrptang #1}%
The space $\Dgrptang$ is naturally the tangent space to the group $\Dgrp$.
Similarly, $\DDer \cdot f$ is
naturally the tangent space of the orbit $\DAut\cdot f$ and $\tang{f}$ is
the tangent space to the orbit $\orbit{f}$ for every $f\in \DP$.

Sometimes it is more convenient to work with
$\DS$ than with $\DP$. For each subspace $W\subseteq \DP$ we may consider the orthogonal space
$W^{\perp} \subseteq \DS$.  Below we describe the linear space
$\perpspace{\tang{f}}$.
For $\sigma\in \DS$ by $\ithpartial{\sigma}{i}$ we denote the $i$-th partial
derivative of $\sigma$. We use the convention that $\deg(0) <
0$.
\begin{prop}[tangent space description]\label{ref:tangentspace:prop}
    Let $f\in \DP$. Then
    \[
        \DDer \cdot f = \spann{ x_i \cdot (\delta\hook f)\ |\ \delta\in \DmmS,\
        i=1, \ldots ,n},\qquad \tang{f} = \DS f + \sum_{i=1}^n \DmmS (x_i\cdot
        f).
    \]
    Moreover
    \[
        \perpspace{\tang{f}} = \left\{ \sigma\in \DS\ |\ \sigma\hook f = 0,\
            \ \forall_i\ \ \deg(\ithpartial{\sigma}{i}\hook f) \leq 0\right\}.
    \]
    Suppose further that $f\in \DP$ is homogeneous of degree $d$. Then  $\perpspace{\tang{f}}$ is
    spanned by homogeneous operators and
    \[
        \perpspace{\tang{f}}_{\leq d} = \left\{ \sigma\in \DS\ |\ \sigma\hook f = 0,\
            \ \forall_i\ \ \ithpartial{\sigma}{i}\hook f = 0 \right\}.
    \]

\end{prop}

\begin{proof}
    Let $D\in \DDer$ and $D_i := D(\Dx_i)$.
    By Proposition~\ref{ref:dualderivation:prop} we have $\Ddual{D}(f) =
    \sum_{i=1}^n x_i\cdot (D_i\hook f)$. For any $\delta\in \DmmS$ we may choose
    $D$ so that $D_i = \delta$ and all other $D_j$ are zero. This proves the
    description of $\DDer \cdot f$.
    Now $\tang{f} = \DS f + \spann{ x_i \cdot (\sigma\hook f)\ |\ \sigma\in \DmmS,\
        i=1, \ldots ,n}$. By Lemma~\ref{ref:commutator:lem} we have
        $x_i(\sigma\hook f) \equiv \sigma\hook(x_i f) \mod \DS f$. Thus
        \[\tang{f} = \DS f + \spann{ \sigma\hook (x_i\cdot f)\ |\ \sigma\in \DmmS,\
        i=1, \ldots ,n} = \DS f + \sum \DmmS (x_i f).
    \]
    Now let $\sigma\in \DS$ be an operator such that $\ip{\sigma}{\tang{f}}
    = 0$. This is equivalent to $\sigma\hook (\tang{f}) = 0$, which
    simplifies to $\sigma\hook f = 0$ and $(\sigma\DmmS)\hook (x_i f) = 0$ for all $i$.
    We have $\sigma\hook (x_if) = x_i (\sigma\hook f) +
    \DPut{tmp}{\ithpartial{\sigma}{i}}\hook f = \Dtmp\hook f$, thus
    we get equivalent conditions:
    \[
        \sigma\hook f = 0\quad \mbox{and} \quad \DmmS\hook (\Dtmp \hook f) = 0,
    \]
    and the claim follows. Finally, if $f$ is homogeneous of degree $d$ and $\sigma\in \DS$
    is homogeneous of degree at most $d$ then $\ithpartial{\sigma}{i}\hook
    f$ has no constant term and so $\deg(\ithpartial{\sigma}{i}\hook f) \leq 0$
    implies that $\ithpartial{\sigma}{i}\hook f = 0$.
\end{proof}

\begin{remark}\label{ref:cotangent:rmk}
    Let $f\in \DP$ be homogeneous of degree $d$.
    Let $i \leq d$ and $K_i := \perpspace{\tang{f}}_{i}$.
    Proposition~\ref{ref:tangentspace:prop} gives a useful connection of $K_i$
    with the conormal sequence.
    Namely, let $I = \Ann{f}$ and $B = \Apolar{f} = \DS/I$. We have $(I^2)_i \subseteq
    K_{i}$ and the quotient space fits into the conormal sequence of $\DS\to
    B$, making that sequence exact:
    \begin{equation}\label{eq:cotangent}
        0\to \left(K/I^2\right)_i\to \left( I/I^2 \right)_i \to
        \left(\Omega_{\DS/k}\tensor
        B\right)_i\to \left(\Omega_{B/k}\right)_{i}\to
        0.
    \end{equation}
    This is expected from the point of view of deformation theory.
    Recall that by \cite[Theorem~5.1]{hartshorne_deformation_theory} the
    deformations of $B$ over $k[\varepsilon]/\varepsilon^2$ are in one-to-one
    correspondence with elements of a $k$-linear space $T^1(B/k, B)$. On the other hand,
    this space fits \cite[Prop~3.10]{hartshorne_deformation_theory} into the sequence
    \def\Hom#1#2{\operatorname{Hom}(#1, #2)}%
    \[
        0\to \Homthree{B}{\Omega_{B/k}}{B} \to
        \Homthree{B}{\Omega_{S/k}\tensor B}{B}\to \Homthree{B}{I/I^2}{B}
        \to T^1(B/k, B)\to 0.
    \]
    Since $B$ is Gorenstein, $\Homthree{B}{-}{B}$ is exact and we have
    $T^1(B/k, B)_i  \simeq \Hom{K/I^2}{B}_i$ for all $i \geq 0$.
    The restriction $i\geq 0$ appears because $\Hom{K/I^2}{B}$ is the tangent
    space to deformations of $B$ inside $\DS$, whereas $T^1(B/k, B)$
    parameterises all deformations.
\end{remark}

\DDef{Autunip}{\operatorname{Aut}^+(\DS)}%
\DDef{grpunip}{\Dgrp^+}
\DDef{Derunip}{\mathfrak{aut}^+}%
Now we introduce a certain subgroup $\Dgrpunip$ of $\Dgrp$.
It plays an important part in characteristic zero, because $\Dgrpunip$-orbits
are closed in $\DP$.
This group is also very useful in applications, because it preserves the top
degree form, which allows induction on the degree type arguments.

Each automorphism of $\DS$ induces a linear
map on its cotangent space: we have a restriction $\DAut\to
\DPut{glcot}{\operatorname{GL}(\DmmS/\DmmS^2)}$. Let us denote by $\DAutunip$
the group of automorphisms which act as identity on the tangent space:
$\DAutunip = \left\{ \varphi\in \DS\ |\ \forall_{i}\ \varphi(\Dx_i) -
\Dx_i\in\DmmS^2 \right\}$. We have the following sequence of groups:
\begin{equation}\label{eq:groups}
    1 \to \DAutunip \to \DAut \to \Dglcot\to 1.
\end{equation}
We define
\[\Dgrpunip = \DAutunip \ltimes (1 + \DmmS) \subseteq \Dgrp.\]
Note that we have the following exact sequence:
\begin{equation}\label{eq:unipradical}
    1\to \Dgrpunip \to \Dgrp \to \Dglcot \times k^*\to 1.
\end{equation}

Correspondingly, let $\DPut{Der}{\mathfrak{aut}}$ denote the
space of derivations preserving $\DmmS$, i.e.~derivations such that $D(\DmmS) \subseteq
\DmmS$.
Let $\DDerunip$ denote the space of derivations
such that $D(\DmmS) \subseteq \DmmS^2$. Denoting by
$\DPut{glliecot}{\mathfrak{gl}\left( \DmmS/\DmmS^2 \right)}$ the space of
linear endomorphisms of $\DmmS/\DmmS^2$, we have we following sequence of linear
spaces:
\begin{equation}\label{eq:derivations}
    0\to \DDerunip \to \DDer \to \Dglliecot\to 0.
\end{equation}
We define
\[
    \DPut{grpuniptang}{\Dgrptang^+} = \DDerunip + \DmmS.
\]
Following the proof of Proposition~\ref{ref:tangentspace:prop} we get the
following proposition.
\def\tangunip#1{\Dgrpuniptang #1}%
\def\orbitunip#1{\Dgrpunip\!\cdot\!#1}%
\begin{prop}\label{ref:uniptangentspace:prop}
    Let $f\in \DP$. Then $\tangunip{f} = \DmmS f + \sum \DmmS^2(x_i f)$
    so that
    \[
        \perpspace{\tangunip{f}} = \left\{ \sigma\in \DS\ |\
            \deg(\sigma\hook f) \leq 0,\ \ \forall_i\ \deg(\ithpartial{\sigma}{i}\hook f)\leq 1\right\}.
    \]
    If $f$ is homogeneous of degree $d$ then $\tangunip{f}$ is spanned
    by homogeneous polynomials and
    \[
        \perpspace{\tangunip{f}}_{<d} = \left\{ \sigma\in \DS\ |\
            \sigma\hook f = 0,\ \ \forall_i\ \ithpartial{\sigma}{i}\hook f = 0
        \right\} = \perpspace{\tang{f}}_{<d}.
    \]\qed
\end{prop}

\subsection{Characteristic zero}\label{ssec:liegroups}

For simplicity we restrict to $k = \mathbb{C}$ in this section, to freely
speak about Lie groups and algebras. As a general reference we suggest
\cite{fultonharris}.
One technical problem is that the group $\DAut$ is infinite-dimensional. In
this section we implicitly replace $\DS$ with $\DS/\DmmS^{D}$ and $\DAut$ with
the group $\operatorname{Aut}(\DS/\DmmS^{D})$, where $D\gg 0$ is larger than
the degree of any element of $\DP$ considered.
Then $\DAut$ and $\Dgrp$ become Lie groups.

Note that the group $\Dgrpunip\subseteq \operatorname{GL}(\DS)$ is
\emph{unipotent}: the map $\varphi - \operatorname{id}$ is nilpotent for every
$\varphi\in\Dgrpunip$. Therefore, we get the following
theorem.
\begin{thm}\label{ref:closedorbits}
    For every $f\in \DP$ the orbit $\orbitunip{f}$ is closed in $\DP$,
    in both Euclidean and Zariski topologies.
\end{thm}
Classically, the Lie algebra corresponding to $\DAut$ is $\DDer$. Moreover by
the exact sequences~\eqref{eq:groups} and~\eqref{eq:derivations}, the Lie
algebra
of $\DAutunip$ is $\DDerunip$. Finally, the Lie algebra of $\Dgrp$ is
$\DDer + \DS = \Dgrptang$ and the Lie algebra of $\Dgrpunip$ is $
\DDerunip + \DmmS = \Dgrpuniptang$.
In particular
\[
    \dim \orbit{f} = \dim \left(\DDer \cdot f + \DS f\right) =
    \dim \left(\DS f + \sum \DmmS(x_i f)\right).
\]
We now give another proof of Elias-Rossi theorem on canonically graded
algebras, in the following Example~\ref{ex:1nn1Lie}.

\begin{example}[compressed cubics, using Lie theoretic ideas]\label{ex:1nn1Lie}
    Assume that $k = \mathbb{C}$.

    Let $f = f_3 + f_{\leq 2}\in\DP_{\leq 3}$ be an element of degree three
    such $H_{\Apolar{f}} = (1, n, n, 1)$, where $n = H_{\DS}(1)$. Such a
    polynomial $f$ is called
    compressed, see \cite{iaCompressed, EliasRossi_Analytic_isomorphisms} and
    Section~\ref{sec:compressed} for a definition. Then $H_{\Apolar{f}}$ is
    symmetric, thus $H_{\Apolar{f_3}} =
    (1, n, n, 1)$ as explained in Section~\ref{ssec:AGA}.

    We claim that there is an element $\varphi$ of $\Dgrp$
    such that $\Ddual{\varphi}(f_3) = f$. This proves that $\Apolar{f}
    \simeq \Apolar{f_{3}} = \gr\Apolar{f}$. We say that the
    apolar algebra of $f$ is \emph{canonically graded}.

    In fact, we claim that already $\orbitunip{f_3}$ is the whole
    space:
    \[\orbitunip{f_3} = f_3 + \DP_{\leq 2}.\]

    From the explicit formula in
    Proposition~\ref{ref:dualautomorphism:prop} we see that $\orbitunip{f_3}
    \subseteq f_3 + \DP_{\leq 2}$. Then it is a Zariski closed
    subset. To prove equality it enough to prove that these spaces have the same
    dimension, in other words that $\tangunip{f_3} = \DP_{\leq 2}$.

    Let $\sigma\in \perpspace{\tangunip{f_3}}_{\leq 2}$ be non-zero. By
    Proposition~\ref{ref:uniptangentspace:prop} we get that $\sigma\hook
    f_3 = 0$ and $\ithpartial{\sigma}{i} \hook f_3 = 0$ for all $i$.
    Then there is a degree one operator annihilating $f_3$. But this
    contradicts the fact that $H_{\Apolar{f_3}}(1) = n = H_{\DS}(1)$.
    Therefore $\perpspace{\tangunip{f_3}}_{\leq 2} = 0$ and the claim
    follows.
\end{example}

\section{Applications}\label{sec:applications}

In this part of the paper we give several new classification results using
the above theory. The two most
important ones are the $t$-compressed algebras in
Section~\ref{sec:compressed}, and algebras having Hilbert
function $(1, 3, 3, 3, 1)$ in Example~\ref{ex:13331}. We begin and end with some
examples computable by hand, which are nevertheless useful, see
in particular the reference in Example~\ref{ex:poweroflinear}.

We make the following assumption
\begin{assumption}
    The field $k$ is algebraically closed. It has characteristic zero or
    greater than the degree of all considered elements of $\DP$.
\end{assumption}
For example, when we analyse cubics, we assume that $k$ has characteristic
different from $2$ or $3$. When analysing degree $d$ polynomials, as in
Example~\ref{ex:poweroflinear}, we assume that
the characteristic is at least $d+1$ or $0$. The fact that $k$ is
algebraically closed is assumed because of various geometrical arguments
involved, such as in Proposition~\ref{ref:topdegreecomp:prop}, and because of
assumptions required for the references: Example~\ref{ex:13331}
refers to \cite{LO} and \cite{cjn13} which both work over an algebraically
closed field. It seems plausible that the classification is feasible over
non-algebraically closed fields, however then there will be more
isomorphism types: already the classifications of homogeneous quadrics
over $\mathbb{C}$ and $\mathbb{R}$ differ.
In contrast,
the assumption that $k$ has large enough characteristic is hardly avoidable
since key results are false without it, see
Example~\ref{ex:1nn1char2} and compare with Example~\ref{ex:1nn1}. This
assumption is frequently used to guarantee that a non-constant element of $\DS$
has a non-zero derivative, which is needed to effectively apply
Proposition~\ref{ref:tangentspace:prop}, see
Proposition~\ref{ref:sp:nosmallorder:prop}.

Under Assumption, we may always think of $\DP$ as a polynomial ring, see
Proposition~\ref{ref:divpowssmallarepoly:prop}.
Before we begin, we make the following observation.
As shown by Equation~\eqref{eq:unipradical}, the group $\Dgrp$ is build from
$\Dgrpunip$, $\Dglcot$ and $k^*$. We will be mostly
concerned with $\Dgrpunip$, in particular we will analyse its orbits. This
often suffices to prove important results as in the cubic case above
(Example~\ref{ex:1nn1}) or below for $t$-compressed algebras.
However sometimes the action of $\Dgrpunip$ is not sufficient and the
difference between this group and $\Dgrp$ becomes the substance of the result.
This is the case for algebras having Hilbert function $(1, 3, 3, 3, 1)$.
The group $\Dgrpunip$ is also convenient because it allows degree-by-degree
analysis which is our main technical tool, see
Proposition~\ref{ref:topdegreecomp:prop}.

When classifying Artinian Gorenstein quotients of $\DS$ with Hilbert function $H$ we will
usually assume that $\dim \DS = H(1)$. This does not reduce generality and
simplifies the argument a bit.

\subsection{Basic examples and tools}

In this section we analyse the $\Dgrp$-orbits of ``easy'' homogeneous
polynomials: those with with small Waring
and border Waring rank.

\begin{example}[Rank one]\label{ex:poweroflinear}
    The simplest example seems to be the case $f = \ell^d$ for some
    linear form $\ell$. We may assume $\ell = x_1$ and apply
    Proposition~\ref{ref:uniptangentspace:prop} to see that
    \[ \perpspace{\tangunip{x_1^d}}_{<d} = (\Dx_2, \ldots ,\Dx_n)^2,
        \qquad
        \tangunip{x_1^d}= \spann{x_1^r m\ |\ m\in \DP_{\leq 1}, r + 1 < d},
    \]
    or, in an invariant form, $\tangunip{\ell^d} = \spann{\ell^r \cdot m\ |\ m\in \DP_{\leq 1},\ r + 1<d}$.
    Using Proposition~\ref{ref:dualautomorphism:prop} one can even compute
    the orbit itself.
    For example, when $d = 4$, the orbit is equal to
    \[
        \orbitunip{\ell^4} = \left\{ \ell^4 + \ell^2\cdot m_1 + \ell\cdot m_2 + m_3 + m_1^2\ |\ m_i\in
            \DP_{\leq 1}\right\}.
    \]
    This example together with Corollary~\ref{ref:leadingformremoval:cor}
    plays an important role in~the paper \cite{cjn13}, compare \cite[Lemma~4.2,
    Example~4.4, Proposition~5.13]{cjn13}.
\end{example}

\def\Dy{\beta}%
\begin{example}[Border rank two]\label{ex:borderintwovars}
    Consider $f = x^{d-1}y$.
    Assume $\DS = k[[\Dx, \Dy]]$ and $\DP = k_{dp}[x, y]$.
    As above, the apolar ideal of $f$ is
    monomial, equal to $(\Dx^d, \Dy^2)$. Using
    Proposition~\ref{ref:uniptangentspace:prop} we easily get that
    \[
        \perpspace{\tangunip{f}}_{<d} = (\Dy^3)_{<d},
    \]
    so $\tangunip{f}$ is spanned by monomials $x^{a}y^{b}$, where $b\leq 2$ and $a
    + b < d$. Note that in contrast with Example~\ref{ex:poweroflinear}, the
    equality $\perpspace{\tangunip{f}}_{<d} = \Ann{f}^2_{<d}$ does
    \emph{not} hold. This manifests the fact that $\Apolar{f}$ is a
    deformation of $\Apolar{x^d + y^d}$.
\end{example}

Before we analyse the rank two polynomials i.e. $x^d + y^d$, we need a few
more observations.

\def\tdf#1{\operatorname{tdf}(#1)}%
Let for every $g\in \DP$ the symbol $\tdf{g}$ denote the top degree form of
$g$,~so that for example $\tdf{x_1^3 + x_2^2x_3 + x_4^2} = x_1^3 + x_2^2x_3$.
\begin{prop}\label{ref:topdegreecomp:prop}
    Let $f\in \DP$. Then the top degree form of every element of $\orbitunip{f}$ is equal to the top degree form of $f$.
    Moreover,
    \begin{equation}\label{eq:tdf}
        \left\{ \tdf{g - f}\ |\ g\in \orbitunip{f} \right\} = \left\{ \tdf{h}\ |\ h\in
    \tangunip{f} \right\}.
    \end{equation}
    If $f$ is homogeneous, then both sides of~\eqref{eq:tdf} are equal to
    the set of homogeneous elements of $\tangunip{f}$.
\end{prop}

\begin{proof}
    Consider the $\DS$-action on
    $\DP_{\leq d}$. This action descents to an $\DPut{Strunc}{\DS/\DmmS^{d+1}}$ action.
    Further in the proof we implicitly replace $\DS$ by $\DStrunc$, thus also
    replacing $\DAut$ and $\Dgrp$ by appropriate truncations.
    Let $\varphi\in \Dgrpunip$. Since $(\operatorname{id} - \varphi)\left( \DmmS^i \right) \subseteq \DmmS^{i+1}$ for
    all $i$, we have $(\operatorname{id} - \varphi)^{d+1} = 0$. By our global
    assumption on the characteristic $D :=
    \log(\varphi)$ is well-defined and $\varphi = \exp(D)$.
    We get an injective map $\exp:\Dgrpuniptang \to
    \Dgrpunip$ with left inverse $\log$. Since $\exp$ is algebraic we see by
    dimension count that its image is open in $\Dgrpunip$. Since $\log$ is
    Zariski-continuous, we get that $\log(\Dgrpunip) \subseteq \Dgrpuniptang$, then
    $\exp:\Dgrpuniptang \to \Dgrpunip$ is an isomorphism.

    Therefore
    \[
        \Ddual{\varphi}(f) = f + \sum_{i=1}^{d}
        \frac{\left(\Ddual{D}\right)^i(f)}{i!} = f + \Ddual{D}(f) +
        \left(\sum_{i=1}^{d-1} \frac{\left(\Ddual{D}\right)^i}{(i+1)!}\right) \Ddual{D}(f).
    \]
    By Remark~\ref{ref:lowersdegree:rmk} the derivation $D\in \Dgrpuniptang$
    lowers the degree, we see that $\tdf{\Ddual{\varphi}f} =
    \tdf{f}$ and
        $\tdf{\Ddual{\varphi}f - f} = \tdf{\Ddual{D}(f)}$. This proves~\eqref{eq:tdf}.
    Finally, if $f$ is homogeneous then $\tangunip{f}$ is equal to the $\spann{ \tdf{h}\ |\ h\in
    \tangunip{f}}$ by
    Proposition~\ref{ref:uniptangentspace:prop}, and the last claim
    follows.\qedhere

    For an elementary proof, at least for the subgroup $\DAutunip$,
    see~\cite[Proposition~1.2]{Matczuk_unipotent_derivations}.
\end{proof}

The following almost tautological
Corollary~\ref{ref:leadingformremoval:cor} enables one to prove that a
given apolar algebra is canonically graded inductively, by lowering the degree
of the remainder.
\begin{cor}\label{ref:leadingformremoval:cor}
    Let $F$ and $f$ be polynomials. Suppose that the leading
    form of $F - f$ lies in $\tangunip{F}$. Then there is an element $\varphi\in
    \Dgrpunip$ such that $\deg(\Ddual{\varphi}f - F) < \deg(f - F)$.
\end{cor}

\begin{proof}
    Let $G$ be the leading form of $f - F$ and $e$ be its degree.
    By Proposition~\ref{ref:topdegreecomp:prop} we may find $\varphi\in
    \DAutunip$ such that $\tdf{\Ddual{\varphi}(F) - F} = -G$, so that
    $\Ddual{\varphi}(F) \equiv F - G \mod \DP_{\leq e-1}$.
    By the same proposition we have $\deg(\Ddual{\varphi}(f - F) - (f -
    F)) < \deg(f - F) = e$, so that $\Ddual{\varphi}(f - F) \equiv f - F \mod
    \DP_{\leq e-1}$. Therefore
    $\Ddual{\varphi}(f) -F = \Ddual{\varphi}(F) + \Ddual{\varphi}(f - F) -F \equiv f
    - G - F \equiv 0 \mod \DP_{\leq e-1}$, as claimed.
\end{proof}

\begin{example}[Rank two]\label{ex:xsys}
    \def\Dy{\beta}%
    Let $P = k_{dp}[x, y]$, $S = k[[\Dx, \Dy]]$ and $F = x^d + y^d$ for some $d\geq
    2$. Then $H_{\Apolar{F}} = (1, 2, 2, 2, \ldots , 2, 1)$. We
    claim that the orbit $\orbitunip{F}$ consists precisely of polynomials $f$ having leading form
    $F$ and such that $H_{\Apolar{f}} = H_{\Apolar{F}}$.

    \DDef{sigx}{\ithpartial{\sigma}{x}}%
    \DDef{sigy}{\ithpartial{\sigma}{y}}%

Let us first compute $\perpspace{\tangunip{F}}$. Since $F$ is homogeneous,
    Proposition~\ref{ref:uniptangentspace:prop} shows that
    \[
        \perpspace{\tangunip{F}}_{<d} = \left\{ \sigma\in \DS_{<d}\ |\ \Dsigx\hook
    F = \Dsigy \hook F = \sigma\hook F = 0\right\}.
    \]
    Since $\Ann{F} = (\Dx\Dy, \Dx^d - \Dy^d)$, we see that
    \[
        \perpspace{\tangunip{F}}_{<d} = (\Dx\Dy)^2_{<d}.
    \]

    Now we proceed to the description of $\orbitunip{F}$.
    It is clear that every $f\in \orbitunip{F}$ satisfies the conditions given.
    Conversely, suppose that $f \in \DP$ satisfies these conditions: the
    leading form of $f$ is $F$ and $H_{\Apolar{f}} = H_{\Apolar{F}}$. If
    $d = 2$, then the claim follows from Example~\ref{ref:examplequarics:ex},
    so we may assume $d\geq 3$. By applying $\Dx^{d-2}$ and $\Dy^{d-2}$ to
    $f$, we see that $x^2$ and $y^2$ are leading forms of partials of $f$. Since $H_{\Apolar{f}}(2) =
    2$, these are the only leading forms of partials of degree two.

    Let $G$ be the leading form of $f - F$. Suppose that $G$ contains a
    monomial $x^ay^b$, where $a, b \geq 2$. Then $\Dx^{a-1}\Dy^{b-1}\hook f =
    xy + l$, where $l$ is linear, then we get a contradiction with the
    conclusion of the previous paragraph.

    Since $G$ contains no monomials of the form $x^ay^b$ with $a, b \geq
    2$, we see that $G$ is annihilated by $\perpspace{\tangunip{F}}$,
    so it lies in $\tangunip{F}$.
    By Corollary~\ref{ref:leadingformremoval:cor} we may find $u\in
    \Dgrpunip$ such
    that $\deg(u f - F) < \deg(f - F)$.
    Thus, replacing $f$ by $uf$ we lower the degree of $f - F$.
    Repeating, we arrive to the point where $f - F = 0$, so that $f = F$.
\end{example}

The analysis made in Example~\ref{ex:xsys} may be generalized
to obtain the following Proposition~\ref{ref:squareofideal:prop}.

\def\dimk{\dim_k\,}%
\begin{prop}\label{ref:squareofideal:prop}
    Let $f\in \DP$ be a polynomial with leading form $F$. Let $I = \Ann{F}$.
    Fix an integer $t\geq 0$ and assume that
    \begin{enumerate}
        \item $\dimk\Apolar{f} = \dimk\Apolar{F}$.
        \item we have $\perpspace{\tangunip{F}}_i = I^2_i$ for
            all $i$ satisfying $t \leq i \leq d-1$.
    \end{enumerate}
    Then there is an element $F + g\in \orbitunip{f}$ such that $\deg(g) < t$.
    Equivalently, $\Apolar{f}  \simeq \Apolar{F + g}$ for some polynomial $g$
    of degree less than $t$.
\end{prop}

\begin{proof}
    We apply induction with respect to $\deg(f - F)$.
    If $\deg(f - F) < t$ then we are done.
    Otherwise, it is
    enough to find $u\in \Dgrpunip$ such that $\deg(uf - F) < \deg(f - F)$.

    \def\gr{\operatorname{gr}\,}%
    Since $\Apolar{F}$ is a quotient of $\gr \Apolar{f}$, the first condition
    implies that
    \[\Apolar{F}  \simeq  \gr \Apolar{f},\]
    Thus for every element
    $i\in I = \Ann{F}$ we may find $\sigma\in \DS$ such that $(i + \sigma)\hook f = 0$ and
    $\ord{\sigma} > \ord{i}$.

    Let $G$ be
    the leading form of $f - F$ and $t \leq r \leq d-1$ be its degree. We will
    now prove that $G\in \tangunip{F}$. By
    assumption, it is enough to show that $I^2_r$ annihilates $G$. The ideal
    $I$ is homogeneous so it is enough to show that for any
    elements $i, j\in I$ such that $\deg(ij) = r$ we have
    \[
        (ij)\hook G = 0.
    \]
    Take $\sigma\in \DS$ such that $(j - \sigma)\hook f = 0$. Then
    $\ord{i\sigma} > \ord{ij} = r = \deg(f-F)$, thus $(i\sigma)\hook f = (i\sigma)\hook F =
    0$. Therefore,
    \[
(ij)\hook G = (ij)\hook (F + G) = (ij)\hook f = i(j - \sigma)\hook f = 0.
    \]
    and the claims follow.
\end{proof}

\begin{example}[Rank $n$]\label{ex:rankn}
    Let $f\in \DP = k_{dp}[x_1, \ldots ,x_n]$ be a polynomial of degree $d\geq 4$ with
    leading form $F = x_1^{d} +  \ldots  + x_{n}^d$ and Hilbert function
    $H_{\Apolar{f}} = (1, n, n, n,  \ldots , n, 1)$. Then the apolar algebra of $f$
    is isomorphic to an apolar algebra of $F + g$, where $g$ is a
    polynomial of degree at most three.

    Indeed, let $I = \Ann{F}$. Then the generators of $I_{<d}$ are $\alpha_i
    \alpha_j$ for $i\neq j$, thus the ideals $I_{<d}$ and
    $(I^2)_{<d}$ are monomial. Also the ideal $\perpspace{\tangunip{F}}$ is
    monomial by its description from
    Proposition~\ref{ref:uniptangentspace:prop}.
    The only monomials of degree at least $4$ which do not lie in $I^2$ are of
    the form $\alpha_i^{d-1}\alpha_j$ and these do not lie in
    $\perpspace{\tangunip{F}}$.
    Therefore $I^2_{i} = \perpspace{\tangunip{F}}_i$ for all $4\leq i\leq d-1$.
    The assumptions of Proposition~\ref{ref:squareofideal:prop} are satisfied
    with $t = 4$ and the claim follows.

    It is worth noting that for every $g$ of degree at most three the apolar
    algebra of $F + g$ has Hilbert
    function $(1, n, n,  \ldots , n, 1)$, so the above considerations give
    a full classification of polynomials in $\DP$ with this Hilbert function and leading
    form $F$.

    Since $\perpspace{\tangunip{F}}_3 = \spann{\alpha_i \alpha_j \alpha_l\ |\
    i<j<l}$, we may also assume that $f = F + \sum_{i<j<l} \lambda_{ijl}
    x_ix_jx_l$ for some $\lambda_{ijl}\in k$, see also Example~\ref{ex:13331}.
\end{example}

\subsection{Compressed algebras and
generalisations}\label{sec:compressed}

In this section we use Proposition~\ref{ref:tangentspace:prop} to obtain a
generalisation of wonderful results of Elias and Rossi on being
canonically graded (see Corollary~\ref{ref:eliasrossi:cor})
and answer their question~stated in
\cite[Remark~3.6]{EliasRossi_Analytic_isomorphisms}.

Example~\ref{ex:1nn1} is concerned with a degree three polynomial
$f$ such that the Hilbert function of $\Apolar{f}$ is maximal i.e. equal to
$(1, n, n, 1)$ for $n = H_{\DS}(1)$.  Below we generalise the results obtained
in this example to polynomials of arbitrary degree.

Recall that a local Artinian Gorenstein algebra $A$ of socle degree $d$ (see
Section~\ref{ssec:AGA}) is called \emph{compressed} if
\[H_A(i) = \min\left(H_{\DS}(i), H_{\DS}(d - i)\right) =
    \min\left(\binom{i+n-1}{i}, \binom{d-i+n-1}{d-i}
    \right)\quad\mbox{for all}\ \
i=0, 1,  \ldots , d.\]
Here we introduce a slightly more general notation.

\begin{defn}[$t$-compressed]\label{ref:compressed:def}
    Let $A = S/I$ be a local Artinian Gorenstein algebra of socle degree $d$. Let
    $t\geq 1$.
    Then $A$ is called \emph{$t$-compressed} if the following conditions are
    satisfied:
    \begin{enumerate}
        \item $H_A(i) = H_{\DS}(i) = \binom{i+n-1}{i}$ for all $0\leq i \leq t$,
        \item $H_{A}(d-1) = H_{\DS}(1)$.
    \end{enumerate}
\end{defn}

\begin{example}
    Let $n = 2$. Then $H_A = (1, 2, 2, 1, 1)$ is not $t$-compressed, for any
    $t$. The function $H_A = (1, 2, 3, 2, 2, 2, 1)$ is $2$-compressed. For any
    sequence $*$ the function $(1, 2, *, 2, 1)$ is $1$-compressed.
\end{example}

Note that it is always true that $H_{A}(d-1) \leq H_{A}(1) \leq H_{\DS}(1)$,
thus both conditions above assert that the Hilbert function is maximal
possible. Therefore they are open in $\DP_{\leq d}$.

\def\floor#1{\left\lfloor #1 \right\rfloor}%
\begin{remark}\label{ref:compressedtrivia:rmk}
    The maximal value of $t$, for which $t$-compressed algebras exists, is $t
    = \left \lfloor d/2\right\rfloor$.  Every compressed algebra is
$t$-compressed for $t = \floor{d/2}$ but not
    vice versa. If $A$ is graded, then $H_A(1) = H_A(d-1)$, so the condition $H_{A}(d-1) = H_{\DS}(1)$
    is satisfied automatically.
\end{remark}

The following technical Remark~\ref{ref:socleminusone:rmk} will be useful later. Up to some extent, it
explains the importance of the second condition in the definition of
$t$-compressed algebras.
\begin{remark}\label{ref:socleminusone:rmk}
    \def\DmmA{\mathfrak{m}_A}%
    Let $A = \Apolar{f}$ be a $t$-compressed algebra with maximal ideal $\DmmA$. We have
    $\dim \DP_{\leq 1} = H_{A}(d-1) + H_{A}(d) = \dim \DmmA^{d-1}/\DmmA^d + \dim \DmmA^d =
    \dim \DmmA^{d-1}$. Moreover $\DmmA^{d-1}  \simeq
    \DmmS^{d-1} f$ as linear spaces and $\DmmS^{d-1} f \subseteq
    \DP_{\leq 1}$. Thus
    \[
        \DmmS^{d-1} f = \DP_{\leq 1}.
    \]
\end{remark}

The definition of $t$-compressed algebras explains itself in the following
Proposition~\ref{ref:sp:nosmallorder:prop}.
\begin{prop}\label{ref:sp:nosmallorder:prop}
    Let $f\in \DP$ be a polynomial of degree $d\geq 3$ and $A$ be its
    apolar algebra.
    Suppose that $A$ is $t$-compressed.
    Then the $\Dgrpunip$-orbit of $f$ contains $f + \DP_{\leq t+1}$. In particular
    $f_{\geq t+2} \in \orbitunip{f}$, so that $\Apolar{f}  \simeq \Apolar{f_{\geq t+2}}$.
\end{prop}

\begin{proof}
\DDef{tmplowdeg}{\DP_{\leq t+1}}%
    First we show that $\Dtmplowdeg\subseteq\tangunip{f}
    $,~i.e. that no non-zero operator of order at most $t+1$ lies in
    $\perpspace{\tangunip{f}}$.
    Pick such an operator. By Proposition~\ref{ref:uniptangentspace:prop} it
    is not constant. Let $\sigma'$ be any of its non-zero partial
    derivatives. Proposition~\ref{ref:uniptangentspace:prop} asserts that
    $\deg(\sigma'\hook f)\leq 1$.
    Let $\ell := \sigma'\hook f$.
    By Remark~\ref{ref:socleminusone:rmk} every linear polynomial is contained
    in $\DmmS^{d-1} f$. Thus we may choose a $\delta\in \DmmS^{d-1}$ such that
    $\delta\hook f = \ell$. Then $(\sigma' - \delta)\hook f = 0$. Since
    $d\geq 3$, we have $d - 1 > \floor{d/2} \geq t$, so that $\sigma - \delta$ is an operator of order
    at most $t$ annihilating $f$. This contradicts the fact that $H_{A}(i) =
    H_{\DS}(i)$ for all $i\leq t$. Therefore $\Dtmplowdeg\subseteq\tangunip{f}$.

    Second, pick a polynomial $g\in f + \DP_{\leq t+1}$. We prove that $g\in
    \orbitunip{f}$ by induction on $\deg(g - f)$.
    The top degree form of $g - f$ lies
    in $\tangunip{f}$. Using Corollary~\ref{ref:leadingformremoval:cor} we
    find $\varphi\in \Dgrpunip$ such that $\deg(\Ddual{\varphi}(g) - f) <
    \deg(g - f)$.
\end{proof}

For completeness, we state the following consequence of the previous result.
\begin{cor}\label{ref:compressedpoly:cor}
    Let $f\in \DP$ be a polynomial of degree $d\geq 3$ and $A$ be its
    apolar algebra. Suppose that $A$ is compressed. Then $A  \simeq
    \Apolar{f_{\geq \floor{d/2} + 2}}$.
\end{cor}

\begin{proof}
    The algebra $A$ is
    $\floor{d/2}$-compressed and the claim follows from
    Proposition~\ref{ref:sp:nosmallorder:prop}.
\end{proof}

As a corollary we reobtain the result of Elias and Rossi, see \cite[Thm
3.1]{EliasRossi_Analytic_isomorphisms}.
\begin{cor}\label{ref:eliasrossi:cor}
    Suppose that $A$ is a compressed Artinian Gorenstein local $k$-algebra of
    socle degree $d\leq 4$. Then $A$ is canonically graded~i.e. isomorphic to
    its associated graded algebra $\operatorname{gr} A$.
\end{cor}

\begin{proof}
    The case $d\leq 2$ is easy and left to the reader. We assume $d\geq 3$, so
    that $3\leq d\leq 4$.

    Fix $n = H_A(1)$ and choose $f\in \DP = k_{dp}[x_1, \ldots ,x_n]$ such
    that $A  \simeq \Apolar{f}$. This is possible by the existence of standard
    form, see~\cite[Thm~5.3AB]{ia94} or \cite{iaPedro}.
    Let $f_d$ be the top degree part of $f$. Since $\floor{d/2} + 2 = d$,
    Corollary~\ref{ref:compressedpoly:cor} implies that $f_d\in \orbitunip{f}$. Therefore the
    apolar algebras of $f$ and $f_d$ are isomorphic.
    The algebra $\Apolar{f_d}$ is a quotient of $\operatorname{gr}
    \Apolar{f}$. Since $\dimk \operatorname{gr}\Apolar{f} = \dimk \Apolar{f} =
    \dimk \Apolar{f_d}$ it follows that
    \[
        \Apolar{f}  \simeq \Apolar{f_{d}}  \simeq \operatorname{gr}\Apolar{f},
    \]
    which was to be proved.
\end{proof}

The above Corollary~\ref{ref:eliasrossi:cor} holds under the assumptions that
    $k$ is algebraically closed and of characteristic not equal to $2$ or $3$. The
    assumption that $k$ is algebraically closed is unnecessary as proven for
    cubics in Example~\ref{ex:1nn1}, the cases of quartics is similar.
    Surprisingly, the assumption on the characteristic is necessary.

    \begin{example}[compressed cubics in characteristic two]\label{ex:1nn1char2}
        Let $k$ be a field of characteristic two.
        Let $f_3\in \DP_{3}$ be a cubic form such that $H_{\Apolar{f_3}} = (1,
        n, n, 1)$ and $\Dx_1^2\hook f_3 = 0$.
        Then there is a degree three polynomial $f$ with leading form $f_3$,
        whose apolar algebra is compressed but not canonically graded.

        Indeed, take $\sigma = \Dx_1^2$. Then all derivatives of $\sigma$ are
        zero because the characteristic is two. By Proposition~\ref{ref:tangentspace:prop} the element
        $\sigma$ lies in $\perpspace{\tang{f_3}}$. Thus $\tang{f_3}$ does not
        contain $\DP_{\leq 2}$ and so $\orbit{f_3}$ does not contain $f_3 +
        \DP_{\leq 2}$. Taking any $f\in f_3 + \DP_{\leq 2}$ outside the orbit
        yields the desired polynomial. In fact, one can explicitly check that
        $f = f_3 + \DPel{x_i}{2}$ is an example.

        A similar example shows that over a field of characteristic three
        there are compressed quartics which are not canonically graded.
    \end{example}

\paragraph{Compressed algebras of socle degree $d \geq 5$ in two variables.}
    As noted in \cite[Example~3.4]{EliasRossi_Analytic_isomorphisms}, the claim of
    Corollary~\ref{ref:eliasrossi:cor} is false for $d = 5$. Below we explain
    this from the point of view of our theory. First, we give an example of
    a compressed algebra of socle degree five, which is not canonically
    graded.

\begin{example}[$H_A = (1, 2, 3, 3, 2, 1)$, special]\label{ex:123321special}
    Let $n = 2$, $\DP = k_{dp}[x_1, x_2]$ and $\DS = k[[\Dx_1, \Dx_2]]$. Take $F =
    x_1^3x_2^2\in \DP$ and $A = \Apolar{F} = \DS/(\Dx_1^4, \Dx_2^3)$. Since
    $x_1^2, x_1x_2, x_2^2$ are all partials of $F$
    we have $H_{A}(2) = 3$ and the algebra $A$ is compressed.
    It is crucial to note that
    \begin{equation}\label{eq:badelementinstable}
        \Dx_2^4\in \perpspace{\tang{F}}.
    \end{equation}
    Indeed, the
    only nontrivial derivative of this element is $4\Dx_2^3\in \Ann{F}$,
    so~\eqref{eq:badelementinstable} follows from
    Proposition~\ref{ref:tangentspace:prop}.
    Therefore $\orbit{F}$ is strictly contained in $F + \DP_{\leq 4}$. Pick
    any element $f\in
    F + \DP_{\leq 4}$ not lying in $\orbit{F}$.
    The associated graded of $\Apolar{f}$ is equal to $\Apolar{F}$ but
    $\Apolar{f}$ is \emph{not} isomorphic to $\Apolar{F}$, thus
    $\Apolar{f}$ is not canonically graded.
    In fact one may pick $f = F + x_2^4$, so that $\Apolar{f} = \DS/(\Dx_1^4,
    \Dx_2^3 - \Dx_1^3\Dx_2)$.
\end{example}

    On the positive side, Corollary~\ref{ref:compressedpoly:cor} implies that
    every compressed algebra of socle degree $5$ is
    isomorphic to $\Apolar{f}$ for some $f = f_5 + f_4$. Thus we may always
    remove the cubic part of $f$.
    But even more is true: a \emph{general} compressed algebra of socle degree
    five is canonically graded, as shown in Example~\ref{ex:123321general}
    below.

\begin{example}[$H_A = (1, 2, 3, 3, 2, 1)$, general]\label{ex:123321general}
    Let $f\in \DP = k_{dp}[x_1, x_2]$ be a \emph{general} polynomial of degree
    five with respect to the natural affine space structure on $k_{dp}[x_1, x_2]_{\leq 5}$. Then
    $\Apolar{f}$ and $\Apolar{f_5}$ are compressed  by \cite{iaCompressed}. The
    ideal $\Ann{f_5}$ is a complete intersection of a cubic and a quartic.
    Since we assumed that $f$ is general, we may also assume that the cubic
    generator $c$ of $\Ann{f_5}$ is not a power of a linear form.

    We claim that $\Apolar{f}$ is canonically graded. Equivalently, we claim
    that $f\in \orbit{f_5}$. It is enough to show that
    \begin{equation}\label{eq:fivecompressed}
        \orbitunip{f_5} = f_5 + \DP_{\leq 4}.
    \end{equation}
    We will show that $\perpspace{\tangunip{f_5}}_{\leq 4} = 0$.
    This space is spanned by homogeneous elements. Suppose
    that this space contains a non-zero homogeneous $\sigma$.
    Proposition~\ref{ref:uniptangentspace:prop} implies that $\sigma$ is
    non-constant and that all
    partial derivatives of $\sigma$ annihilate $f_5$. These derivatives have
    degree at most three, thus they are multiples of the cubic generator $c$ of
    $\Ann{f_5}$. This implies that all partial derivatives of $\sigma$ are
    proportional, so that $\sigma$ is a power of a linear form. Then also $c$
    is a power of a linear form, which contradicts earlier assumption. We
    conclude that no non-zero $\sigma\in \perpspace{\tangunip{f_5}}_{\leq
    4}$ exists.
    Now, the equation~\eqref{eq:fivecompressed} follows from
    Corollary~\ref{ref:leadingformremoval:cor}.
\end{example}

\begin{example}[general polynomials in two variables of large
    degree]\label{ex:12345large}
    Let $F\in \DP = k_{dp}[x_1, x_2]$ be a homogeneous form of degree $d\geq 9$ and
    assume that no linear form annihilates $F$.
    The ideal $\Ann{F} = (q_1, q_2)$ is a complete intersection. Let $d_i :=
    \deg q_i$ for $i=1, 2$ then $d_1 + d_2 = d + 2$. Since $d_1, d_2 \geq 2$,
    we have $d_1, d_2 \leq d$.

    We claim that the apolar algebra of a general (in the sense explained in
    Example~\ref{ex:123321general}) polynomial $f\in \DP$ with
    leading form $F$ is not canonically graded.
    Indeed, we shall prove that $\orbit{F} \cap (F + \DP_{<d})$ is strictly contained in $F +
    \DP_{<d}$ which is the same as to show that
    $K := \perpspace{\tang{F}}_{<d}$ is non-zero.

    The sequence~\eqref{eq:cotangent} from Remark~\ref{ref:cotangent:rmk}
    becomes
    \[
        0 \to (K/I^2)_{d-1} \to \left(A[-d_1] \oplus A[-d_2]\right)_{d-1}\to
        \left(A\oplus A\right)_{d-2},
    \]
    where the rightmost map has degree $-1$. Note that
    \[
        \dim \left(A[-d_1] \oplus A[-d_2]\right)_{d-1} = (d-1-d_1 + 1) +
        (d-1-d_2+1) = 2d - (d_1 + d_2) = d - 2.
    \]
    By assumption $d - 2 > 6 = \dim \left(A\oplus
        A\right)_{d-2}$,
    so that $(K/I^2)_{d-1} \neq 0$ and our claim follows.
\end{example}

\paragraph{The set of canonically graded algebras is constructible but not
necessarily open nor closed.}

Again consider $\DP_{\leq d}$ as an affine space with Zariski topology.
Suppose that there exists a polynomial of degree $d$ whose apolar algebra is not canonically
graded. Elias and Rossi asked in
\cite[Remark~3.6]{EliasRossi_Analytic_isomorphisms} whether in this case also the
apolar algebra of a \emph{general} polynomial of this degree is not canonically
graded. The following proposition answers this question negatively.

\begin{prop}\label{ref:nopennclosed:prop}
    Fix the degree $d$ and the number of variables $n = \dim
        \DS$.  Consider
        the set $\mathcal{G} \subset \DP_{\leq d}$ of dual socle generators of canonically
    graded algebras of socle degree $d$. This set is irreducible and
    constructible, but in general it is neither open nor closed.
\end{prop}
\begin{proof}
        The set $\mathcal{G} = \Dgrp \cdot \DP_{d}$ is the image of $\Dgrp
        \times \DP_{d}$, thus irreducible
and constructible. Example~\ref{ex:123321special} and
Example~\ref{ex:123321general} together show that
for $n=2$, $d = 5$ this set is not closed. Example~\ref{ex:12345large} shows
that for $n=2$, $d \geq 9$ this set is not open.
\end{proof}


    In view of Proposition~\ref{ref:nopennclosed:prop} it is natural
    to ask for which degrees and numbers of variables the
    answer to the question above is positive; the apolar algebra of a \emph{general}
    polynomial in $\DP_{\leq d}$ is a canonically graded algebra.
    Example~\ref{ex:123321general} shows that this happens for
    $n = 2$, $d = 5$. However, the list of cases where the answer may be positive is
    short, as we shall see now.

    \begin{prop}\label{ref:generalcangradrarely:prop}
        Assume that $k$ is a field of characteristic zero.
        Suppose that $n = \dim \DS$ and $d$ are such that the apolar algebra of a general
    polynomial in $\DP_{\leq d}$ is canonically graded. Then $(n, d)$ belongs
    to the following list:
    \begin{enumerate}
        \item $d\leq 4$ and $n$ arbitrary,
        \item $d = 5$ and $n\leq 6$,
        \item $d = 6$ and $n = 2$,
        \item $d$ arbitrary and $n = 1$.
    \end{enumerate}
    \end{prop}

    \begin{proof}
        Fix $n$, $d$ outside the above list and suppose that $\Dgrp\cdot
        \DP_{d}$ is dense in $\DP_{\leq d}$. Since we are over a field of
        characteristic zero, the tangent map $\Dgrp \times \DP_{d} \to
        \DP_{\leq d}$ at a general point $(g, F)$ is surjective. Then by
        $\Dgrp$-action, this map is surjective also at $(1, F)$.
        Its image is $\tang{F} + \DP_{d}$. Thus
        \begin{equation}\label{eq:rubbsh}
            \tang{F} + \DP_{d} = \DP_{\leq d}.
        \end{equation}
        Let us now analyse $\tang{F}$.
        By Proposition~\ref{ref:tangentspace:prop} the equations of $\tang{F}$
        are given by $\sigma\in \DS$ satisfying
        $\sigma\hook F = 0$ and
        \begin{equation}
            \ithpartial{\sigma}{i}\hook F = 0\quad\mbox{for all}\quad i=1,
            \ldots ,n.
            \label{eq:allderivs}
        \end{equation}

        Consider $\sigma\in \DS$ homogeneous of degree $d-1$. The space
        $\Ann{F}_{d-2}$ is of codimension $H_{\Apolar{F}}(d-2) =
        H_{\Apolar{F}}(2) \leq \binom{n+1}{2}$. Thus for fixed $i$ the condition
        $\ithpartial{\sigma}{i}\hook F = 0$ amounts to at most
        $\binom{n+1}{2}$ linear conditions on the coefficients of $\sigma$.
        Summing over $i=1, \ldots ,n$ we get $n\cdot \binom{n+1}{2}$
        conditions.
        Now, if $(n, d)$ is outside the list, then
        \[
            n\cdot \binom{n+1}{2} < \binom{n-1+d-1}{d-1},
        \]
        so that there exists a non-zero $\sigma$
        satisfying~\eqref{eq:allderivs}. Since $\sigma = (d-1)^{-1}\sum
        \Dx_i\ithpartial{\sigma}{i}$ also $\sigma\hook F = 0$. Thus
        $\sigma\in \perpspace{\tang{F}}$ is a non-zero element.
        Therefore $\tang{F}$ does not contain the whole $\DP_{\leq d-1}$,
        which contradicts~\eqref{eq:rubbsh}.
    \end{proof}

    We have some computational evidence that in all cases listed in
    Proposition~\ref{ref:generalcangradrarely:prop} a general polynomial indeed gives
    a canonically graded algebra. Of course for $d\leq 4$ it is the result of
    Elias and Rossi reproved in Corollary~\ref{ref:eliasrossi:cor}.
    We put forward the following conjecture.

    \begin{conjecture}\label{ref:density:conj}
        Assume that $k$ is a field of characteristic not equal to $2, 3, 5$.
        The pairs $(n, d)$ such that the apolar algebra of a general
    polynomial in $\DP_{\leq d}$ is canonically graded are precisely the pairs
    listed below:
    \begin{enumerate}
        \item $d\leq 4$ and $n$ arbitrary,
        \item $d = 5$ and $n\leq 6$,
        \item $d = 6$ and $n = 2$.
        \item $d$ arbitrary and $n = 1$.
    \end{enumerate}
    Here $n = \dim \DS$ is the number of variables.
    \end{conjecture}

    We have some supporting evidence in all cases $d = 5$, $n\leq 6$ and $d = 6$,
    $n=2$. Namely, for a pseudo-randomly chosen form $F$
    a machine computation of $\perpspace{\tang{F}}$ via
    Proposition~\ref{ref:tangentspace:prop} reveals that
    \[\tang{F} \supset  \DP_{\leq d-1}.\] Thus arguing by semicontinuity we get that
    the map $\Dgrp \times \DP_{d} \to \DP_{\leq d}$ is dominating and the
    conjecture follows. However we find
    the use of computer here far from satisfactory, since it dims the reasons
    for this result. Hopefully someone would invent a computer-free proof.

\paragraph{Improvements using symmetric decomposition.}

The condition $H_{A}(d-1) = n$ in definition of $t$-compressed algebras may be
slightly weakened, which is sometimes useful. The price one pays is a more
technical assumption. We present the result below.
An example of use of Proposition~\ref{ref:sp:improvenosmallorder:prop}
is given in Example~\ref{ex:1222111}.
For information on the symmetric decomposition $\Delta_{r}$ of the Hilbert function,
see \cite{ia94}, \cite{cjn13} or \cite[Section~2]{BJMR}. In short, the Hilbert
function $H$ of an apolar algebra of socle degree $d$ admits a canonical decomposition $H =
\sum_{i=0}^{d} \Delta_i$, where $\Delta_i$ is a vector of length
$d-i$ which is symmetric: $\Delta_i(j) = \Delta_i(d-i-j)$ for all $j$.
\begin{prop}\label{ref:sp:improvenosmallorder:prop}
    Let $f\in \DP$ be a polynomial of degree $d\geq 3$ and $A$ be its
    apolar algebra. Let $\Delta_{\bullet}$ be the symmetric decomposition of
    $H_A$.
    Suppose that
    \begin{enumerate}
        \item $H_A(r) = H_{\DS}(r) = \binom{r + n - 1}{r}$ for all $0\leq r \leq t$.
        \item $\Delta_{r}(1) = 0$ for all $d - 1 - t \leq r$.
    \end{enumerate}
    Then the $\Dgrpunip$-orbit of $f$ contains $f + \DP_{\leq t+1}$. In particular
    $f_{\geq t+2} \in \orbitunip{f}$.
\end{prop}
\begin{proof}
In the notation of \cite[Section~2]{BJMR} or \cite{cjn13} we see that
\def\Cfilt#1{C_{#1, 1}(1)}%
\[
    \Delta_a(1) = \dim \Cfilt{a}/\Cfilt{a-1},\quad \mbox{where}\quad
    \Cfilt{a} =
    \left\{\mbox{ linear polynomials in }\DmmS^{d - 1 - a}f \right\}.
\]
Moreover, the assumption $H_A(1) = n$ guarantees that $\Cfilt{d-2} =
\DP_{\leq 1}$ is the full space of linear polynomials. The assumption $\Delta_{r}(1) = 0$ for
all $d - 1 - t\leq r \leq d - 2$ shows that
\begin{equation}\label{eq:highdegreeops}
    \DP_{\leq 1} = \Cfilt{d-2} = \Cfilt{d-3}=  \ldots = \Cfilt{d-2 - t}=
    \DmmS^{t+1} f.
\end{equation}
Thus, for every $\ell\in \DP_{\leq 1}$ we have a $\delta$ of order greater
than $t$ such that $\delta\hook f = \ell$.

Now we repeat the proof of Proposition~\ref{ref:sp:nosmallorder:prop} with one
difference: instead of referring to Remark~\ref{ref:socleminusone:rmk} we
use~Equation \eqref{eq:highdegreeops} to obtain, for a given linear form $\ell$, an
element $\delta\in \DS$ of order greater than $t$ and such that $\delta\hook f
= \ell$.
\end{proof}

\subsection{Further examples}

Below we present some more involved examples, which employ the
tools developed in the previous section.

\begin{example}[Hilbert function $(1, 2, 2, 1)$]\label{ex:1221}
    By the result of Elias and Rossi (Proposition~\ref{ref:eliasrossi:cor}) every apolar algebra
    with Hilbert function $(1, 2, 2, 1)$ is canonically graded. Then it is
    isomorphic to $\Apolar{x^3 + y^3}$ or $\Apolar{x^2y}$ and these algebras
    are not isomorphic.

    The example was treated in \cite[p.~12]{Casnati_Notari_6points} and \cite[Prop~3.6]{EliasRossiShortGorenstein}.
\end{example}

\begin{example}[Hilbert function $(1, 2, 2, 2, 1)$]\label{ex:12221}
    Let $f\in \DP_{\leq 4}$ be a polynomial such that $H_{\Apolar{f}} = (1, 2,
    2, 2, 1)$. If $f_4 = x^4 + y^4$ then if fact $\Apolar{f}  \simeq
    \Apolar{f_4}$ by Example~\ref{ex:xsys}, so we may assume $f_4 =
    x^3y$. Since $f$ is $1$-compressed we may also
    assume $f_{\leq 2} = 0$, whereas by
    Corollary~\ref{ref:leadingformremoval:cor} together with
    Example~\ref{ex:borderintwovars} we may assume $f_3 = cy^3$ for $c\in k$. Thus
    \[
        \Apolar{f} \simeq \Apolar{x^3y + cy^3}.
    \]
    By multiplying variables by suitable constants, we may assume $c = 0$ or $c
    =1$ and obtain three possibilities:
    \[
        f = x^4 + y^4,\qquad f = x^3y,\qquad f = x^3y + y^3.
    \]
    Note that the three
    types appearing above are pairwise non-isomorphic. Indeed, an isomorphism
    may only occur between $\Apolar{x^3y}$ and $\Apolar{x^3y + y^3}$.
    Suppose that such exists, so that there is $\varphi\in \Dgrp$ such that
    $\Ddual{\varphi}(x^3y) = x^3y + y^3$. Write $\varphi = gu$, for $g$
    linear
    and $u\in \Dgrpunip$. Then $g$ preserves $x^3y$ and a direct
    check shows that $g$ is diagonal. Then
    $\Ddual{u}(x^3y) = x^3y + cy^3$ for some non-zero $c\in k$ and $y^3 =
    c^{-1}\tdf{u(x^3y) - x^3y}
    \in \tangunip{(x^3y)}$, by Proposition~\ref{ref:topdegreecomp:prop}. This is a
    contradiction with Example~\ref{ex:borderintwovars}.

    This example was analysed, among others,
        in~\cite{Casnati__Isomorphisms_types_of_aG_to_nine}, where Casnati
    classifies all Artinian Gorenstein algebras of length at most $9$. See
    especially \cite[Theorem~4.4]{Casnati__Isomorphisms_types_of_aG_to_nine}.
\end{example}

\begin{example}[Hilbert function (1, 2, 2,  \ldots , 2, 1),
    general]\label{ex:122lotsof21}
    Consider the set of polynomials $f\in \DP = k_{dp}[x, y]_{\leq d}$ such
    that $H_{\Apolar{f}} = (1, 2,  \ldots , 2, 1)$
    with $d-1$ twos occurring. This set is irreducible in Zariski topology,
    by~\cite[Proposition~4.8]{cjn13} which uses \cite[Theorem
    3.13]{iarrobino_punctual}. The leading form of a general element of this
    set has, up to a coordinate change, the form $x^d + y^d$.
    Example~\ref{ex:xsys} shows that the apolar algebra of such an element is
    canonically graded.
\end{example}

\begin{remark}\label{ref:isomorphismtypesintwovars:rmk}
    In general, there are $d-1$ isomorphism types of
    almost-stretched algebras of socle degree $d$ and with Hilbert function
    $(1, 2, 2,  \ldots , 2, 1)$ as proved
    in~\cite[Theorem~4.4]{Casnati__Isomorphisms_types_of_aG_to_nine}; see also
    \cite[Remark~5., p.~447]{EliasVallaAlmostStretched}.
    The claim is recently generalised by Elias and Homs,
    see~\cite{EliasHoms}.
\end{remark}

\begin{example}[Hilbert function $(1, 3, 3, 3, 1)$]\label{ex:13331}
    Consider now a polynomial $f\in \DP = k_{dp}[x, y, z]$ whose Hilbert function is
    $(1, 3, 3, 3, 1)$.
    Let $F$ denote the leading form of $f$.
    By \cite{LO} or \cite[Prop~4.9]{cjn13} the form $F$ is linearly equivalent to
    one of the following:
    \[
        F_1 = x^4 + y^4 + z^4,\quad F_2 = x^3y + z^4,\quad F_3 = x^2(xy + z^2).
    \]

    Since $\Apolar{f}$ is $1$-compressed, we have
    $\Apolar{f}  \simeq \Apolar{f_{\geq 3}}$; we may assume that the quadratic part
    is zero. In fact by the explicit description of top degree form in
    Proposition~\ref{ref:topdegreecomp:prop} we see that
    \[\orbitunip{f} = f + \tangunip{F} + P_{\leq 2}.\]
    Recall that $\Dgrp/\Dgrpunip$ is the product of the group of linear
    transformations and $k^*$ acting by multiplication.

    \paragraph{The case $F_1$.} Example~\ref{ex:rankn} shows that
    $\perpspace{\tangunip{F}}_{\leq 3}$ is spanned by $\alpha\beta\gamma$.
    Therefore we may assume $f = F_1 + c\cdot xyz$ for some $c\in k$. By
    multiplying variables
    by suitable constants and then multiplying whole $f$ by a constant, we may
    assume $c = 0$ or $c = 1$. As before, we get
    two non-isomorphic algebras. Summarising, we got two isomorphism types:
    \[
        f_{1,0} = x^4 + y^4 + z^4,\quad f_{1,1} = x^4 + y^4 + z^4 + xyz.
    \]
    Note that $f_{1, 0}$ is canonically graded, whereas $f_{1,1}$ is a
    complete intersection.

    \paragraph{The case $F_2$.} We have $\Ann{F_2}_2 = (\alpha\gamma,
    \beta^2, \beta\gamma)$, so that $\perpspace{\tangunip{F_2}}_{\leq
    3} = \spann{\beta^3, \beta^2\gamma}$. Thus we may assume $f =
    F_2 + c_1 y^3 + c_2 y^2z$. As before, multiplying $x$, $y$ and $z$ by
    suitable constants we may assume $c_1, c_2\in \{0, 1\}$. We get four
    isomorphism types:
    \[
        f_{2, 00} = x^3y + z^4,\quad f_{2, 10} = x^3y + z^4 + y^3,\quad f_{2,
        01} =
        x^3y + z^4 + y^2z,\quad f_{2, 11} = x^3y + z^4 + y^3 + y^2z.
    \]
    To prove that the apolar algebras are pairwise non-isomorphic one shows
    that the only linear maps preserving $F_2$ are diagonal and argues as an
    in Example~\ref{ex:12221} or as described in the case of $F_3$ below.

    \paragraph{The case $F_3$.} We have $\Ann{F_3}_2 = (\beta^2,\,
    \beta\gamma,\,\alpha\beta - \gamma^2)$ and
    \[\perpspace{\tangunip{F_3}}_{\leq 3} = \spann{\beta^2\gamma,\ \beta^3,\
        \alpha\beta^2 -
    2\beta\gamma^2}.\]
    We may choose $\spann{y^3, y^2z, yz^2}$ as the complement of $\tangunip{F_3}$
    in $P_3$. Therefore the apolar algebra of each $f$ with top degree form $F_3$ is
    isomorphic to the apolar algebra of
    \[
        f_{3, *} = x^3y + x^2z^2 + c_1y^3 + c_2y^2z + c_3yz^2
    \]
    and two distinct such polynomials $f_{3, *1}$ and $f_{3, *2}$ lie in
    different $\Dgrpunip$-orbits.
    We identify the set of $\Dgrpunip$-orbits with
    $P_3/\tangunip{F_3}  \simeq \spann{y^3, y^2z, yz^2}$.
    We wish to determine isomorphism classes, that is,
    check which such $f_{3, *}$ lie in the same $\Dgrp$-orbit. A little care
    should be taken here, since $\Dgrp$-orbits will be bigger than expected.

    Recall that $\Dgrp/\Dgrpunip  \simeq \Dglcot \times k^*$ preserves the
    degree. Therefore, it is enough to look at the operators stabilising
    $F_3$. These are $c\cdot g$, where $c\in k^*$ is a constant and $g\in
    \Dglcot$ stabilises $\spann{F_3}$, i.e.~$\Ddual{g}(\spann{F_3}) =
    \spann{F_3}$. Consider such a $g$.
    It is a linear automorphism of $\DP$ and maps $\Ann{F}$ into itself.
    Since $\beta(\lambda_1\beta +
    \lambda_2\gamma)$ for $\lambda_i\in k$ are the only reducible quadrics in $\Ann{F_3}$ we see that
    $g$ stabilises $\spann{\beta, \gamma}$, so that $\Ddual{g}(x) = \lambda x$ for a
    non-zero $\lambda$. Now it is straightforward to check directly that the
    group of linear maps stabilising $\spann{F_3}$ is generated by the following
    elements
\def\DtmpA{a}%
\def\DtmpB{b}%
    \begin{enumerate}
        \item homotheties: for a fixed $\lambda\in k$ and for all linear
            $\ell\in \DP$ we have $\Ddual{g}(\ell) = \lambda \ell$.
        \item for every $\DtmpA , \DtmpB \in k$ with $\DtmpB \neq 0$,
            the map $\DPut{tab}{t_{\DtmpA , \DtmpB }}$ given by
            \[
                \Dtab(x) = x,\ \ \Dtab(y) = -\DtmpA ^2x + \DtmpB ^2 y - 2\DtmpA \DtmpB  z,\ \
                \Dtab(z) = \DtmpA x + \DtmpB z.
            \]
            which maps $F_3$ to $\DtmpB ^2 F_3$.
    \end{enumerate}
    The action of $\Dtab$ on $P_3/\tangunip{F_3}$ in the basis $(y^3, y^2z, yz^2)$
    is given by the matrix
    \[
        \begin{pmatrix}
            \DtmpB ^6 & 0 & 0\\
            -6\DtmpA \DtmpB ^5 & \DtmpB ^5 & 0\\
            \frac{27}{2}\DtmpA ^2\DtmpB ^4 & -\frac{9}{2}\DtmpA \DtmpB ^4 & \DtmpB ^4\\
        \end{pmatrix}
    \]
    Suppose that $f_{3, *} = x^3y + x^2z^2 + c_1y^3 + c_2y^2z + c_3yz^2$ has
    $c_1 \neq 0$. The above matrix shows that we may choose $\DtmpA $ and
    $\DtmpB $ and a homothety $h$ so that
    \[(h\circ \Dtab)(f_{3, *}) = c(x^3y + x^2z^2 + y^3 + c_3 yz^2),\quad
    \mbox{where}\ c\neq 0,\ c_3\in \{0, 1\}.\]
    Suppose $c_1 = 0$. If $c_2 \neq 0$ then we may choose $\DtmpA $,
    $\DtmpB $ and $\lambda$ so
    that $(h\circ \Dtab)(f_{3, *}) = x^3y + x^2z^2 + y^2z$. Finally, if $c_1 = c_2 =
    0$, then we may choose $\DtmpA = 0$ and $\DtmpB$, $\lambda$ so that $c_3 = 0$ or $c_3 = 1$. We get
    at most five isomorphism types:
    \begin{align*}
        &f_{3, 100} = x^3y + x^2z^2 + y^3,\quad f_{3, 101} = x^3y + x^2z^2 +
        y^3 + yz^2,\quad f_{3, 010} = x^3y + x^2z^2 + y^2z,\\
        &f_{3, 001} =
        x^3y + x^2z^2 + yz^2,\quad f_{3, 000} = x^3y + x^2z^2.
    \end{align*}
    By using the explicit description of the $\Dgrp$ action on $P_3/\tangunip{F_3}$
    one checks that the apolar algebras of the above polynomials are pairwise
    non-isomorphic.

    \paragraph{Conclusion:} There are $11$ isomorphism types of algebras with
    Hilbert function $(1, 3, 3, 3, 1)$.
    We computed the tangent spaces to the corresponding orbits in
    characteristic zero, using a computer implementation of the description in
    Proposition~\ref{ref:tangentspace:prop}.
    The dimensions of the orbits are as follows:
\begin{center}

\begin{tabular}{@{}l l c @{}}
    orbit && dimension\\ \midrule
    $\orbit{(x^4 + y^4 + z^4 + xyz)}$   && $29$\\
    $\orbit{(x^4 + y^4 + z^4)}$ && $28$\\
    $\orbit{\left( x^3y + z^4 + y^3 + y^2z \right)}$ && $28$\\
    $\orbit{\left( x^3y + z^4 + y^3 \right)}$ && $27$\\
    $\orbit{\left(x^3y + z^4 + y^2z\right)}$ && $27$\\
    $\orbit{\left(x^3y + z^4\right)}$ && $26$\\
\end{tabular}
\hspace{1cm}\begin{tabular}{@{}l l c @{}}
    orbit && dimension\\ \midrule
    $\orbit{\left(x^3y + x^2z^2 + y^3 + yz^2 \right)}$ && $27$\\
    $\orbit{\left(x^3y + x^2z^2 + y^3 \right)}$ && $26$\\
    $\orbit{\left( x^3y + x^2z^2 + y^2z \right)}$ && $26$\\
    $\orbit{\left( x^3y + x^2z^2 + yz^2 \right)}$ && $25$\\
    $\orbit{\left( x^3y + x^2z^2\right)}$ && $24$\\
    &&
\end{tabular}
\end{center}
\def\DGLthree{\operatorname{GL}_3}
    The closure of the orbit of $f_{1, 1} = x^4 + y^4 + z^4 + xyz$ is
    contained in $\DGLthree(x^4 + y^4 + z^4) + \DP_{\leq 3}$, which is
    irreducible of dimension $29$. Since the orbit itself has dimension $29$
    it follows that it is dense inside. Hence the orbit closure contains
    $\DGLthree(x^4 + y^4 + z^4) + \DP_{\leq 3}$.
    Moreover, the set $\operatorname{GL}_{3}(x^4 + y^4 + z^4)$ is dense inside
    the set $\sigma_3$ of forms $F$ whose apolar algebra has Hilbert function
    $(1, 3, 3, 3, 1)$.
    Thus the orbit of $f_{1, 1}$ is dense inside the set of polynomials with
    Hilbert function $(1, 3, 3, 3, 1)$. Therefore, the latter set is
    irreducible and of dimension $29$.

    It
    would be interesting to see which specializations between different
    isomorphism types are possible. There are some obstructions. For example,
    the $\DGLthree$-orbit of $x^3y + x^2z^2$ has smaller dimension than the
    $\DGLthree$-orbit of $x^3y + z^4$. Thus $x^3y + x^2z^2 + y^3 + yz^2$ does
    not specialize to $x^3y+z^4$ even though its $\Dgrp$-orbit has higher
    dimension.
\end{example}

\begin{example}[Hilbert function $(1, 3, 4, 3, 1)$]\label{ex:13431}
    \def\Dz{\gamma}%
    Consider now a polynomial $f\in k_{dp}[x, y, z]$ whose apolar algebra has
    Hilbert function $H = (1, 3, 4, 3, 1)$. Let $F$ be the leading form of
    $f$. Since $H$ is symmetric, the apolar algebra of $F$ also has Hilbert
    function $H$. In particular it is annihilated by $2$-dimensional space of
    quadrics. Denote this space by $Q$.
    Suppose that the quadrics in $Q$ do not share a common factor, then they
    are a complete intersection. By looking at the Betti numbers, we see that
    $\Ann{F}$ itself is a complete intersection of $Q$ and a cubic.
    From this point of view, the classification of $F$ using our ideas seems
    ineffective compared to classifying ideals directly.
    Therefore we will not attempt a full classification. Instead we show that
    there are infinitely many isomorphism types and discuss non-canonically
    graded algebras.

    Let $V = V(Q) \subset
    \mathbb{P}^2$ be the zero set of $Q$ inside the projective space with
    coordinates $\Dx, \Dy, \Dz$.

    \paragraph{Infinite family of isomorphism types.}
    Consider
    \begin{equation}\label{eq:Flambda}
        F = F_{\lambda} = \lambda_1 x^4 + \lambda_2 y^4 + \lambda_3 z^4 +
        \lambda_4(x+y+z)^4.
    \end{equation}
    for some non-zero numbers $\lambda_{i}$.  Then the Hilbert function of
    $\Apolar{f}$ is $(1, 3, 4, 3, 1)$ and $V(Q)$ is a set of four points, no
    three of them collinear.  Suppose that $F_{\lambda}$ and $F_{\lambda'}$
    are in the same $\Dgrp$-orbit. By Remark~\ref{ref:graded:rem} there is an
    element of $\Dglcot$ mapping $F_{\lambda}$ to $F_{\lambda'}$. Such
    element stabilizes $V(Q)$, which is a set of four points, no three
    of them collinear. But the only elements of $\Dglcot$
    stabilizing such set of four points are the scalar matrices.
    Therefore we conclude that the set of isomorphism classes of $F_{\lambda}$
    is the set of quadruples $\lambda_{\bullet}$ up to homothety. This set is
    in bijection with $(k^*)^4/k^*  \simeq  (k^*)^3$, thus infinite.
    It fact the set of $F_{\lambda}$ with $\lambda_4 = 1$ is a threefold in
    the moduli space of finite algebras with a fixed basis,
    see~\cite{poonen_moduli_space} for construction of this space.

    \paragraph{Non-canonically graded algebras.}
    We now classify forms $F$ such that all polynomials $f$ with leading form $F$ lie in
    $\orbit{F}$. As in Example~\ref{ex:13331} we see that
    \[
        \orbitunip{F} = F + \tangunip{F} + P_{\leq 2}.
    \]
    Thus we investigate $\perpspace{\tangunip{F}}_3$ using
    Proposition~\ref{ref:uniptangentspace:prop}. Let us suppose
    it is non-zero and pick a non-zero
    element $\sigma\in \perpspace{\tangunip{F}}_3$. Then
    $\ithpartial{\sigma}{i}\in \Ann{F}_2 = Q$ for all $i$. Since $Q$ is
    two-dimensional we see that the derivatives of $\sigma$ are linearly
    dependent. Thus up to coordinate change we may assume that $\sigma\in
    k[\Dx, \Dy]$. If $\sigma$ has one-dimensional space of derivatives, then
    $\sigma = \Dx^3$ up to coordinate change and $\frac{1}{3}\ithpartial{\sigma}{1} =
    \Dx^2$ annihilates $Q$. If $\sigma$ has two-dimensional space of
    derivatives $Q$, then $Q$ intersects the space of pure squares in an
    non-zero element $\Dx^2$, thus $\Dx^2$ annihilates $F$ in this case also.
    Conversely, if $\Dx^2$ annihilates $F$, then $\Dx^3$ lies in
    $\perpspace{\tangunip{F}}_3$.

    Summarizing, $F + \DP_{\leq 3} \subset \orbit{F}$ if and only if no square
    of a linear form annihilates $F$.
    For example, apolar algebras of all
    polynomials with leading form $F_{\lambda}$ from~\eqref{eq:Flambda} are
    canonically graded. On the other hand, for $F = xz^3 + z^2y^2 + y^4$ and
    $f_{\lambda} = F + ax^3$ the apolar algebra of $f_{\lambda}$ is
    canonically graded if and only if $a = 0$.

    Non-canonically graded algebras with Hilbert function $(1, 3, 4, 3, 1)$
    are investigated independently
    in~\cite{masuti_rossi_Artinian_level_algebras_of_socle_degree_four}, where
    more generally Hilbert functions $(1, n, m, n, 1)$ are considered.
\end{example}

Any discussion of isomorphism types would be incomplete without
tackling the example of Hilbert function $(1, 2, 2, 2, 1, 1, 1)$, which is the
smallest example where infinitely many
isomorphism types appear. It is also instructive as a non-homogeneous
example. Strangely enough, the argument is similar to the previous
examples. We will use some standard tools to deal
with the dual socle generator, see~\cite[Chap 3]{cjn13}.
\begin{example}[Hilbert function $(1, 2, 2, 2, 1, 1, 1)$]\label{ex:1222111}
    Consider any $f\in \DP$ such that
    \[H_{\Apolar{f}} = (1, 2, 2, 2, 1, 1, 1).\]
    Then $f$ is of degree $6$. Using the standard form, see
    \cite[Theorem~5.3]{ia94} or \cite[Section 3]{cjn13}, we may assume, after a
    suitable change of $f$,
    that
    \[
        f = x^6 + f_{\leq 4}.
    \]
    If $y^4$ appears with non-zero coefficient in $f_{\leq 4}$, we conclude
    that $f\in \orbitunip{(x^6 + y^4)}$ arguing similarly as in
    Example~\ref{ex:xsys}. Otherwise, we may assume that $\Dy^2$ is a
    leading form of an element of an annihilator of $f$, so that the only
    monomials in $f_4$ are $x^4$, $x^3y$ and $x^2y^2$.
    By subtracting a suitable element of $\tangunip{x^6}$ and rescaling
    by homotheties, we may assume
    \[
        f = x^6 + c x^2y^2 + f_{\leq 3},
    \]
    for $c = 0$ or $c = 1$. If $c = 0$, then the Hilbert function of $f$ is
    $(1, *, *, 1, 1, 1, 1)$, which is a contradiction; see
    \cite[Lem~1.10]{ia94} or \cite[Lem~4.34]{jelisiejew_MSc} for details. Thus $c = 1$.

    Consider elements of $\perpspace{\tangunip{(x^6 + x^2y^2)}}$ with order
    at most three. Every non-zero partial
    derivative $\tau$ of such an element $\sigma$ has order at most two and
    satisfies
    $\deg(\tau\hook f) \leq 1$. This is only possible if $\tau_2 = \Dy^2$,
    so that $\sigma_3 = \Dy^3$ up to a non-zero constant. Therefore we may assume
    \[
        f = x^6 + x^2y^2 + \lambda y^3 + f_{\leq 2}\quad\mbox{for some}\quad
        \lambda\in k.
    \]
    The symmetric decomposition of $\Apolar{f}$ is $\Delta_0 = (1, 1, 1, 1, 1,
    1, 1)$, $\Delta_1 = \mathbf{0}$, $\Delta_2 = (0, 1, 1, 1, 0)$, $\Delta_3 =
    \mathbf{0}$,
    $\Delta_4 = \mathbf{0}$. In particular $\Delta_4(1) = 0$, so that by
    Proposition~\ref{ref:sp:improvenosmallorder:prop} we may assume
    $f_{\leq 2} = 0$ and
    \[
        f = f_{\lambda} = x^6 + x^2y^2 +  \lambda y^3.
    \]
    Similarly as in the final part of Example~\ref{ex:12221} we
    see that for every $\lambda\in k$ we get a distinct algebra $A_{\lambda} =
    \Apolar{f_{\lambda}}$
    where $A_{\lambda}  \simeq A_{\lambda'}$ if and only if $\lambda =
    \lambda'$. Thus in total we get $|k| + 1$ isomorphism types.
    Note that the $|k|$ types $f_{\lambda}$ form a curve in the moduli space of finite
    algebras with a fixed basis, see~\cite{poonen_moduli_space} for
    information on this space. The point corresponding to $y^6 + x^4$ does not seem
    to be related to this curve.
\end{example}

\small
\def\cdprime{$''$}

\end{document}